\documentclass[a4paper,11pt,reqno]{amsart}

\pdfoutput=1

\usepackage[ansinew]{inputenc}
\usepackage[english]{babel}
\usepackage{amsmath}
\usepackage{amssymb}
\usepackage{amsthm}
\usepackage{mathrsfs}
\usepackage[all]{xy}
\usepackage{verbatim}
\usepackage{arydshln}

\usepackage{mathtools}
\usepackage{ifthen}
\usepackage{tikz}
\usepackage{todonotes}
\usepackage[final,pdfauthor={Matthew Gelvin, Sune Precht Reeh, Ergun Yalcin},pdftitle={On the basis of the Burnside ring of a fusion system}]{hyperref}
\usepackage{bookmark}

\usepackage[abbrev,msc-links]{amsrefs}

\numberwithin{equation}{section}
\numberwithin{figure}{section}
\numberwithin{table}{section}
\allowdisplaybreaks

\theoremstyle{plain}
\newtheorem{theorem}{Theorem}[section]
\newtheorem{prop}[theorem]{Proposition}
\newtheorem{lemma}[theorem]{Lemma}
\newtheorem{cor}[theorem]{Corollary}

\theoremstyle{definition}
\newtheorem{definition}[theorem]{Definition}

\newtheorem{notation}[theorem]{Notation}

\theoremstyle{remark}
\newtheorem{remark}[theorem]{Remark}

\newtheorem{example}[theorem]{Example}

\newcommand{\ph}{\varphi}

\newcommand{\x}{\times}
\newcommand{\ox}{\otimes}
\renewcommand{\tilde}{\widetilde}

\renewcommand{\bar}{\overline}

\newcommand{\free}[1]{{A^*(#1)}}

\newcommand{\lc}[1]{\prescript{#1\!}{}}

\newcommand{\xra}{\xrightarrow}
\newcommand{\xto}{\xrightarrow}

\def\cF{\mathcal F}

\def\cP{\mathcal P}
\def\cT{\mathcal T}

\def\fB{\mathfrak B}\def\fC{\mathfrak C}

\def\BC{\fB\fC}
\def\tBC{\mathfrak t\fB\fC}

\def\QQ{\mathbb Q}

\def\ZZ{\mathbb Z}

\DeclarePairedDelimiter{\abs}\lvert\rvert
\DeclarePairedDelimiter{\gen}\langle\rangle

\newcommand{\Mark}{\mathord{Mark}}
\newcommand{\FMark}{\mathord{\cF\! Mark}}
\newcommand{\Mob}{\mathord{M\ddot ob}}
\newcommand{\FMob}{\mathord{\cF\! M\ddot ob}}

\usetikzlibrary{arrows,matrix,decorations,positioning,calc}
\tikzset{dot/.style={circle,fill=black,thick,inner sep=0pt,minimum size=1mm,draw}}
\tikzset{arrow/.style={semithick,>=stealth',shorten >=1pt,shorten <=1pt}}
\tikzset{equal/.style={arrow,double distance=2pt}}

\title{On the basis of the Burnside ring of a fusion system}
\author[M. Gelvin]{Matthew Gelvin}
\address{Mathematics and Computer Science  Department, Wesleyan University,
Middletown, CT 06459-0128, U.S.A.}
\email{mgelvin@wesleyan.edu}
\author[S. P. Reeh]{Sune Precht Reeh}
\address{Department of Mathematical Sciences,  University of Copenhagen, Copenhagen, Denmark}
\email{spr@math.ku.dk}
\author[E. Yal\c{c}{\i}n]{Erg\"un Yal\c{c}{\i}n}
\address{Department of Mathematics, Bilkent University,
 06800 Bilkent, Ankara, Turkey}
\email{yalcine@fen.bilkent.edu.tr }
\subjclass[2010]{}
\date{March 24, 2014}

\thanks{The research is supported by the Danish National Research Foundation through the Centre for Symmetry and Deformation (DNRF92). The third author is also partially supported by T\" UB\. ITAK-TBAG/110T712.}

\begin{document}

\begin{abstract}
We consider the Burnside ring $A(\cF)$ of $\cF$-stable $S$-sets for a saturated fusion system $\cF$ defined on a $p$-group $S$. It is shown by S. P. Reeh that the monoid of $\cF$-stable sets is a free commutative monoid with canonical basis $\{\alpha_P\}$.  We give an explicit formula that describes $\alpha_P$ as an $S$-set. In the formula we use a combinatorial concept called broken chains which we introduce to understand inverses of modified M{\" o}bius functions.
\end{abstract}

\maketitle

\section{Introduction}

For a finite group $G$, the Burnside ring $A(G)$ is defined as the Grothendieck ring of the isomorphism classes of $G$-sets with addition given by disjoint union and multiplication by cartesian product. The Burnside ring $A(G)$ is free as an abelian group, with basis given by isomorphism classes of transitive $G$-sets $[G/H]$. In particular the basis elements are in one-to-one correspondence with $G$-conjugacy classes of subgroups of $G$.

One often studies the Burnside ring of a finite group $G$ using the mark homomorphism $\Phi : A(G) \to \ZZ ^{Cl(G)}$, where $Cl(G)$ is the set of $G$-conjugacy classes of subgroups of $G$.  For $K\leq G$, the $K$th coordinate of $\Phi$ is defined by $\Phi_K(X)=|X^K|$ when $X$ is a $G$-set, extended linearly for the rest of $A(G)$.
The ring $\free{G}:= \ZZ ^{Cl(G)}$ is the ring of super class functions $f:Cl(G)\to \ZZ$ with multiplication given by coordinate-wise multiplication. It is called the ghost ring of $G$ and it plays an important role for explaining  $G$-sets using their fixed point data. In particular, it is shown that the mark homomorphism is an injective map with a finite cokernel. This means that using rational coefficients, one can express the idempotent basis of $A^*(G)$ in terms of basis elements $[G/H]$ (see D.~Gluck \cite{Gluck}).

Given a saturated fusion system $\cF$ on a $p$-group $S$, one can define the Burnside ring $A(\cF)$ of the fusion system $\cF$ as a subring of $A(S)$ formed by elements $X\in A(S)$ such that $\Phi _P (X)=\Phi _{\varphi(P)} (X)$ for every morphism $\varphi : P \to S$ in $\cF$. This subring is also the Grothendieck ring of $\cF$-stable $S$-sets (see \eqref{charFstable} for a definition). It is proved by S. P.~Reeh \cite{ReehStableSets} that the monoid of $\cF$-stable $S$-sets is a free commutative monoid with a canonical basis satisfying certain properties.  Our primary interest is to identify the elements of this basis, so we describe it in more detail here.

For every $X\in A(S)$, let $c_Q(X)$ denote the number of $[S/Q]$-orbits in $X$ so that $X=\sum c_Q(X)[S/Q]$, where the sum is taken over the set of $S$-conjugacy classes of subgroups of $S$. For each $\cF$-conjugacy class of subgroups $P$ of $S$, there is a unique (up to $S$-isomorphism) $\cF$-stable set $\alpha_P$ satisfying
\begin{enumerate}
\item $c_Q(\alpha_P)=1$ if $Q$ is fully normalized and $\cF$-conjugate to $P$,
\item $c_Q(\alpha_P)=0$ if $Q$ is fully normalized and not $\cF$-conjugate to $P$.
\end{enumerate}
The set $\{\alpha_P \}$ over all $\cF$-conjugacy classes of subgroups form an (additive) basis for $A(\cF)$ (see Proposition \ref{propDefiningAlphas}).

The main purpose of this paper is to give explicit formulas for the number of fixed points $\abs{(\alpha_P)^Q}$ and for the coefficients $c_Q(\alpha_P)$ of $[S/Q]$-orbits, for the basis element $\alpha_P$. Our first observation is that the matrix of fixed points $\FMark_{Q, P}=\abs{(\alpha_P)^Q}$ can be described using a simple algorithm in linear algebra. We now explain this algorithm.

Let $\Mob=\Mark^{-1}$ denote the inverse matrix of the usual table of marks for $S$. For each $\cF$-conjugacy class of subgroups of $S$, take the sum of the corresponding columns of $\Mob$, obtaining a non-square matrix.  Then, from the set of rows corresponding to an $\cF$-conjugacy class, select one representing a fully $\cF$-normalized subgroup; delete the others.  The resulting matrix $\FMob$ is a square matrix with dimension equal to the number of $\cF$-conjugacy classes of subgroups. Then we observe that the inverse matrix $\FMark:= \FMob^{-1}$ is the matrix of marks for $A(\cF)$. In other words, we prove the following:

\begin{theorem}\label{thmIntroMatrixForAlphas}
Let $\cF$ be a saturated fusion system over a finite $p$-group $S$. Let the square matrix $\FMob$ be constructed as above, with rows and columns corresponding to the $\cF$-conjugacy classes of subgroup in $S$.
Then $\FMob$ is invertible, and the inverse $\FMark:= \FMob^{-1}$ is the matrix of marks for $A(\cF)$, i.e.
\[\FMark_{Q^*,P^*} = \abs*{(\alpha_{P^*})^{Q^*}}.\]
\end{theorem}

Here $Q^*$ and $P^*$ denote the chosen $\cF$-conjugacy class representatives. This theorem is proved as  Theorem \ref{thmMatrixForAlphas} in the paper. We also give a detailed calculation to illustrate this method (see Example \ref{exMatrixmethod}). This is all done in Section \ref{secFixedpoints}.

In Section \ref{secMobiusInversion}  we look closely at the above matrix method and analyze it using M{\" o}bius inversion. We observe that the entries of $\FMark$, the table of marks for $\cF$, can be explained by a combinatorial formula using a concept called (tethered) broken chains (see Definition \ref{defTetheredBrokenChain}). This formula is proved in Theorem \ref{thmFixedPointBrokenChains}.

In Section \ref{secBrokenChains}, we prove the main theorem of the paper, which gives a formula for the coefficients $c_Q(\alpha_P)$ in the linear combination $\alpha _P =\sum c_Q(\alpha_P) [S/Q]$. As in the case of fixed point orders, here also the formula is given in terms of an alternating sum of the number of broken chains linking $Q$ to $P$ (see Definition \ref{defBrokenChain}). The main theorem of the paper is the following:

\begin{theorem}\label{thmIntroOrbitsBrokenChains}
Let $\cF$ be a saturated fusion system over a finite $p$-group $S$. Let $\BC_\cF (Q,P)$ denote the set of $\cF$-broken chains linking $Q$ to $P$.  Then the number of $[S/Q]$-orbits in each irreducible $\cF$-stable set $\alpha_{P}$, denoted $c_Q(\alpha_P)$, can be calculated as
\[c_Q(\alpha_P) = \frac{\abs{W_S P^*}}{\abs{W_S Q}}\cdot \sum_{\sigma\in \BC_\cF(Q,P)} (-1)^{\ell(\sigma)}\]
for $Q,P$ subgroups of $S$, where $P^*\sim_\cF P$ is fully normalized.
\end{theorem}

In the above formula, $\ell(\sigma)$ denotes the length of a broken chain $\sigma =(\sigma_0,\sigma_1,\dotsc,\sigma_k)$ linking $Q$ to $P$ defined as the integer $\ell(\sigma) := k+\abs{\sigma_0}+\dotsb+\abs{\sigma_k}$  (see Definition \ref{defBrokenChain}). This theorem is proved in Section \ref{secBrokenChains} as Theorem \ref{thmOrbitsBrokenChains}. In Example \ref{exBrokenChains}, we illustrate how this combinatorial formula can be used to calculate the coefficients $c_Q(\alpha_P)$ for some subgroups $Q$, $P$ for the fusion system $\cF=\cF_{D_8} (A_6)$.

In Section \ref{secSimplifications}, we prove some simplifications for the formula in Theorem \ref{thmIntroOrbitsBrokenChains}. These simplifications come from observations about broken chains and from properties of M\" obius functions. Then in Section \ref{secApplications} we give an application of our main theorem to characteristic bisets.  Since understanding the characteristic  bisets was one of the motivations for this work, we now say a few more words about this application

Let $S$ be a $p$-group and $\cF$ be a fusion system on $S$ as before. A characteristic biset for the fusion system $\cF$ is an $(S,S)$-biset $\Omega$ satisfying certain properties (see Definition \ref{defCharBiset}). These bisets were first introduced by Linckelmann and Webb, and they play an important role in fusion theory. One of the properties of a characteristic biset is
stability under $\cF$-conjugation, namely for every $\varphi : Q \to S$, the $(Q,S)$-bisets $
\prescript{}{\varphi}\Omega $  and $\prescript{}{Q}\Omega $ are isomorphic. Since each $(S,S)$-biset is a left $(S\times S)$-set, we can convert this stability condition to a stability condition for the fusion system $\cF \times \cF$ on the $p$-group $S\times S$ and consider characteristic bisets as elements in $A(\cF \times \cF)$.

It is shown by M.~Gelvin and S.~P.~Reeh \cite{GelvinReeh} that every characteristic biset includes a unique minimal characteristic biset, denoted by $\Omega_{\min}$. The minimal biset can be described as the basis element $\alpha _{\Delta(S, id)}$ of the fusion system $A(\cF \times \cF)$, where for a morphism $\varphi : Q \to S$ in $\cF$, the subgroup $\Delta (P, \varphi)$ denotes the diagonal subgroup $\{ (\varphi(s),s) \, | \, s\in P \}$ in $S\times S$. Now Theorem \ref{thmIntroOrbitsBrokenChains} can be used to give formulas for the coefficients $c_{\Delta (P, \varphi)} (\Omega _{\min})$. Such formulas are important for various other applications of these bisets (see for example \cite{YalcinUnlu}). Using the new interpretation of these coefficients we were able to give a proof for the statement that
all the stabilizers $\Delta (P, \varphi)$ appearing in $\Omega _{\min}$ must satisfy $P \geq O_p (\cF)$ where $O_p(\cF)$ denotes the largest normal $p$-subgroup of $\cF$. This was originally proved in
\cite{GelvinReeh}*{Proposition 9.11}, the proof we give in Proposition \ref{proApplication} uses broken chains and is much  simpler.

{\bf Acknowledgements:} This work was carried out when the third author was visiting the Centre for Symmetry and Deformation at the University of Copenhagen during the Summer of 2013. He thanks the centre for the financial support which made the visit possible and the director Jesper Grodal for the hospitality that he received during the visit.

\section{Burnside rings for groups and fusion systems}\label{secBurnside}

In this section we recall the Burnside ring of a finite group  $S$  and how to describe its structure in terms of the homomorphism of marks, which embeds the Burnside ring into a suitable ghost ring.
We also recall the Burnside ring of a saturated fusion system $\cF$ on a $p$-group $S$, in the sense of \cite{ReehStableSets}.

Let $S$ be a finite group. We use the letter $S$ instead of $G$ for a finite group since in all the applications of these results the group $S$ will be a $p$-group.  The isomorphism classes of finite $S$-sets form a semiring with disjoint union as addition and cartesian product as multiplication. The Burnside ring of $S$, denoted $A(S)$, is then defined as the additive Grothendieck group of this semiring, and $A(S)$ inherits the multiplication as well. Given a finite $S$-set $X$, we let $[X]$ denote the isomorphism class of $X$ as an element of $A(S)$.
The isomorphism classes $[S/P]$ of transitive $S$-sets form an additive basis for $A(S)$, and two transitive sets $S/P$ and $S/Q$ are isomorphic if and only if the subgroups $P$ and $Q$ are conjugate in $S$.

For each element $X\in A(S)$ we define $c_P(X)$, with $P\leq S$, to be the coefficients when we write $X$ as a linear combination of the basis elements $[S/P]$ in $A(S)$, i.e.
\[X= \sum_{[P]\in Cl(S)} c_P(X) \cdot [S/P],\]
where $Cl(S)$ denotes the set of $S$-conjugacy classes of subgroup in $S$. The resulting maps $c_P\colon A(S) \to \ZZ$ are group homomorphisms, but they are \emph{not} ring homomorphisms.

To describe the multiplication of $A(S)$, it is enough to know the products of basis elements $[S/P]$ and $[S/Q]$. By taking the cartesian product $(S/P)\x (S/Q)$ and considering how it breaks into orbits, one reaches the following double coset formula for the multiplication in $A(S)$:
\begin{equation}\label{eqSingleBurnsideDoubleCoset}
[S/P]\cdot [S/Q] = \sum_{\bar s \in P\backslash S /Q} [S/(P\cap \lc s Q)],
\end{equation}
where $P\backslash S /Q$ is the set of double cosets $PsQ$ with $s\in S$.

Instead of counting orbits, an alternative way of characterising a finite $S$-set is counting the fixed points for each subgroup $P\leq S$. For every $P\leq S$ and $S$-set $X$, we denote the number of $P$-fixed points by $\Phi_{P}(X) := \abs*{X^P}$. This number only depends on $P$ up to $S$-conjugation.
Since we have
\[\abs*{(X \sqcup Y)^P}= \abs*{X^P}+\abs* {Y^P} \quad\text{and}\quad\abs*{(X \x Y)^P}= \abs*{X^P}\cdot\abs*{Y^P}\]
for all $S$-sets $X$ and $Y$, the \emph{fixed point map} $\Phi_{P}$ for $S$-sets extends to a ring homomorphism $\Phi_{P}\colon A(S) \to \ZZ$.
On the basis elements $[S/P]$, the number of fixed points is given by
\begin{equation}\label{eqPhiOnBasis}
\Phi_{Q}([S/P]) = \abs*{(S/P)^Q}= \frac{\abs{N_S(Q,P)}}{\abs P},
\end{equation}
where $N_S(Q,P) = \{s\in S \mid \lc s Q \leq P\}$ is the transporter in $S$ from $Q$ to $P$.
In particular, $\Phi_{Q}([S/P])\neq 0$ if and only if $Q\lesssim_S P$ ($Q$ is subconjugate to $P$).

We have one fixed point homomorphism $\Phi_P$ per conjugacy class of subgroups in $S$, and we combine them into the \emph{homomorphism of marks} \[ \Phi=\Phi^S\colon A(S) \xto{\prod_{[P]} \Phi_{P}} \prod_{[P]\in Cl(S)} \ZZ.\] This ring homomorphism maps $A(S)$ into the product ring $\free{S}:=\prod_{[P]\in Cl(S)} \ZZ$, the so-called \emph{ghost ring} for the Burnside ring $A(S)$.

We think of the elements in the ghost ring $\free{S}$ as superclass functions $Cl(S)\to \ZZ$ defined on the subgroups of $S$ and constant on every $S$-conjugacy class. For an element $\xi\in \free{S}$ we write $\xi (Q)$, with $Q\leq S$, to denote the value of the class function $\xi$ on the $S$-conjugacy class of $Q$. We think of $\xi (Q)$ as the number of $Q$-fixed points for $\xi$, even though $\xi$ might not be the fixed point vector for an actual element of $A(S)$. The ghost ring $\free{S}$ has a natural basis consisting of $e_P$ for each $[P]\in Cl(S)$, where $e_P$ is the class function with value $1$ on the class $[P]$, and $0$ on all the other classes.   The elements $\{e_P \mid [P]\in Cl(S)\}$ are the primitive idempotents of $\free{S}$.

Results by tom Dieck and others show that the mark homomorphism is injective, but not every $\xi \in \free{S}$ is the fixed point vector for an element of $A(S)$. The cokernel of $\Phi$ contains the obstruction to $\xi$ being the fixed point vector of a (virtual) $S$-set, hence we speak of this cokernel as the  \emph{obstruction group} $Obs(S) := \prod_{[P]\in Cl(S)} (\ZZ/\abs{W_S P}\ZZ)$, where $W_S P := N_S P / P$. These statements are combined in the following proposition, the proof of which can be found in \cite{tomDieck}*{Chapter 1}, \cite{Dress}, and \cite{YoshidaSES}.

\begin{prop}\label{propYoshidaGroup}
Let $\Psi=\Psi^S\colon \free{S} \to Obs(S)$ be given by the $[P]$-coordinate functions
\[\Psi_{P}(\xi)  := \sum_{\bar s\in W_S P} \xi \bigl (\gen s P \bigr ) \pmod {\abs{W_S P}}.\]
Then, the following sequence of abelian groups is exact:
\[0\to A(S) \xto{\Phi} \free{S} \xto{\Psi} Obs(S) \to 0.\]
\end{prop}

Note that in the exact sequence above $\Phi$ is a ring homomorphism, but $\Psi$ is just a group homomorphism.

\subsection{The Burnside ring of a saturated fusion system}\label{secBurnsideFusion}

Let $S$ be a finite $p$-group, and suppose that $\cF$ is a saturated fusion system on $S$ (see \cite{AKO} for necessary definitions on fusion systems). We say that a finite $S$-set is $\cF$-stable if the action is unchanged up to isomorphism whenever we act through morphisms of $\cF$. More precisely, if $P\leq S$ is a subgroup and $\ph\colon P\to S$ is a homomorphism in $\cF$, we can consider $X$ as a $P$-set by using $\ph$ to define the action $g.x:=\ph(g)x$ for $g\in P$. We denote the resulting $P$-set by $\prescript{}{P,\ph}X$. In particular when $incl\colon P\to S$ is the inclusion map, $\prescript{}{P,incl}X$ has the usual restriction of the $S$-action to $P$.

Restricting the action of $S$-sets along $\ph$ extends to a ring homomorphism $r_\ph\colon A(S)\to A(P)$, and we let $\prescript{}{P,\ph}X$ denote the image $r_\ph(X)$ for all elements $X\in A(S)$. We say that an element $X\in A(S)$ is \emph{$\cF$-stable} if it satisfies
\begin{equation}\label{charFstable}
\parbox[c]{.9\textwidth}{\emph{$\prescript{}{P,\ph}X=\prescript{}{P,incl}X$ inside $A(P)$, for all $P\leq S$ and homomorphisms $\ph\colon P\to S$ in $\cF$.}}
\end{equation}
The $\cF$-stability condition originally came from considering action maps $S\to \Sigma_X$ into the symmetric group on $X$ that are maps of fusion systems $\cF\to \cF_{\Sigma_X}$.

Alternatively, one can characterize $\cF$-stability in terms of fixed points and the mark homomorphism, and the following three properties are equivalent for all $X\in A(S)$:
\begin{enumerate}
\item\label{itemPhiBurnsideEq}  $X$ is $\cF$-stable.
\item\label{itemPhiStable} $\Phi_{P}(X) = \Phi_{\ph P}(X)$ for all $\ph\in \cF(P,S)$ and $P\leq S$.
\item\label{itemFConjStable} $\Phi_{P}(X) = \Phi_{Q}(X)$ for all pairs $P,Q\leq S$ with $P\sim_{\cF} Q$.
\end{enumerate}
A proof of this claim can be found in \cite{Gelvin}*{Proposition 3.2.3} or \cite{ReehStableSets}. We shall primarily use \ref{itemPhiStable} and \ref{itemFConjStable} to characterize $\cF$-stability.

It follows from property \ref{itemFConjStable} that the $\cF$-stable elements form a subring of $A(S)$. We define the \emph{Burnside ring of $\cF$} to be the subring $A(\cF)\subseteq A(S)$ consisting of all the $\cF$-stable elements. Equivalently, we can consider the actual $S$-sets that are $\cF$-stable: The $\cF$-stable sets form a semiring, and we define $A(\cF)$ to be the Grothendieck group hereof. These two constructions give rise to the same ring $A(\cF)$ -- see \cite{ReehStableSets}.

According to \cite{ReehStableSets}, every $\cF$-stable $S$-set decomposes uniquely (up to $S$-isomorphism) as a disjoint union of irreducible $\cF$-stable sets, where the irreducible $\cF$-stable sets are those that cannot be written as disjoint unions of smaller $\cF$-stable sets. Each irreducible $\cF$-stable set corresponds to an $\cF$-conjugacy class $[P]_\cF = \{Q\leq S \mid \text{$Q$ is isomorphic to $P$ in $\cF$}\}$ of subgroups, and they satisfy the following characterization:

\begin{prop}[\citelist{\cite{ReehStableSets}*{Proposition 4.8 and Theorem A}}]\label{propDefiningAlphas}
Let $\cF$ be a saturated fusion system over $S$.
For each conjugacy class in $\cF$ of subgroups $[P]_\cF$ there is a unique (up to $S$-isomorphism) $\cF$-stable set $\alpha_P$ satisfying
\begin{enumerate}
\item $c_Q(\alpha_P)=1$ if $Q$ is fully normalized and $\cF$-conjugate to $P$,
\item $c_Q(\alpha_P)=0$ if $Q$ is fully normalized and not $\cF$-conjugate to $P$.
\end{enumerate}
The sets $\alpha_P$ form an additive basis for the monoid of all $\cF$-stable $S$-sets. In addition, by construction in \cite{ReehStableSets} the stabilizer of any point in $\alpha_P$ is $\cF$-conjugate to a subgroup of $P$.
\end{prop}

\section{Fixed point orders of the irreducible $\mathcal F$-stable sets}\label{secFixedpoints}

Let $\Mark$ be the matrix of marks for the Burnside ring of $S$, i.e. the matrix for the mark homomorphism $\Phi\colon A(S) \to \free{S}$
with entries
\[\Mark_{Q,P} = \abs*{(S/P)^Q} =\frac{\abs{N_S(Q,P)}}{\abs P}.\]
The rows and columns of $\Mark$ correspond to the $S$-conjugacy classes $[P]_S\in Cl(S)$ of subgroups in $S$. We order the subgroup classes by increasing order of the subgroups, in particular the trivial group $1$ corresponds to the first row and column, and $S$ itself corresponds to the last row and column. This way $\Mark$ becomes upper triangular.

Over the rational numbers the mark homomorphism $\Phi\colon A(S)\ox \QQ \xra\cong \free{S} \ox \QQ$ is an isomorphism, and we let $\Mob=\Mark^{-1}$ be the inverse rational matrix.

From $\Mob$ we construct a further matrix $\FMob$ as follows: For each $\cF$-conjugacy class of subgroups in $S$ we take the sum of the corresponding columns of $\Mob$ to be the columns of $\FMob$. For each $\cF$-conjugacy class of subgroups we choose a fully normalized representative of the class, and then we delete all rows that do not correspond to one of the chosen representatives. The resulting matrix $\FMob$ is guaranteed to be a square matrix with dimension equal to the number of $\cF$-conjugacy classes of subgroups; the rows and columns correspond to the chosen representatives of the $\cF$-conjugacy classes (see Example \ref{exMatrixmethod}).

For each class $[P]_\cF$ let $P^*$ be the chosen representative. The precise description of the entries $\FMob_{Q^*,P^*}$ in terms of $\Mob_{Q,P}$ is then
\[\FMob_{Q^*,P^*} := \sum_{[P]_S\subseteq [P^*]_\cF} \Mob_{Q^*,P}\ .\]

\begin{theorem}\label{thmMatrixForAlphas}
Let $\cF$ be a saturated fusion system over a finite $p$-group $S$. Let the square matrix $\FMob$ be constructed as above, with rows and columns corresponding to the $\cF$-conjugacy classes of subgroup in $S$.
Then $\FMob$ is invertible, and the inverse $\FMark:= \FMob^{-1}$ is the matrix of marks for $A(\cF)$, i.e.
\[\FMark_{Q^*,P^*} = \abs*{(\alpha_{P^*})^{Q^*}}.\]
\end{theorem}

\begin{proof}
In the rational ghost ring $\free{S}\ox \QQ = \prod_{[P]_S} \QQ$ the unit vector $e_P$ is the superclass function with value $1$ for the class $[P]$ and value $0$ for the other subgroup classes. We have one unit vector $e_P$ corresponding to each conjugacy class $[P]_S$, and the matrix of marks $\Mark$ expresses the usual basis for $A(S)\ox \QQ$, consisting of the transitive sets $[S/P]$, in terms of the idempotents $e_P$. Conversely, the inverse $\Mob=\Mark^{-1}$ then expresses the idempotents $e_P$ as (rational) linear combinations of the orbits $[S/P]$.

An element $X\in A(S)\ox \QQ$ is $\cF$-stable if the number of fixed points $\abs*{X^Q}$ and $\abs*{X^P}$ are the same for $\cF$-conjugate subgroups $Q\sim_\cF P$, i.e. if the coefficients of $X$ with respect to the idempotents $e_Q$ and $e_P$ are the same for $\cF$-conjugate subgroups.
The $\cF$-stable elements of $\free{\cF}\ox \QQ\leq \free{S}\ox \QQ$ thus have an idempotent basis consisting of
\[e_P^\cF := \sum_{[P']_S \subseteq [P]_\cF} e_{P'}\ ,\]
with a primitive $\cF$-stable idempotent $e^\cF_P$ corresponding to each $\cF$-conjugacy class of subgroups. To express the idempotent $e_P^\cF$ as a linear combination of orbits $[S/P]$, we just have to take the sum of the columns in $\Mob$ associated to the conjugacy class $[P]_\cF$. Hence counting the number of $[S/Q]$-orbits in $e_P^\cF$, we get
\[c_Q (e_P ^\cF )=\sum_{[P']_S\subseteq [P]_\cF} \Mob_{Q,P'}\ .\]

Let $\cP^*$ denote the set of chosen fully normalized representatives for each $\cF$-conjugacy class of subgroups in $S$. For each $P^*\in\cP ^*$, we have an irreducible $\cF$-stable set $\alpha_{P^*}$, and by Proposition \ref{propDefiningAlphas} any linear combination $X$ of the $\{\alpha_{P^*}\}$ can be determined solely by counting the number of $[S/P^*]$-orbits for each $P^*\in \cP^*$.

Because $e^\cF_P$ is $\cF$-stable, it is a (rational) linear combination of the $\{\alpha_{Q^*}\}$. The coefficients of this linear combination, coincide with the number of $[S/Q^*]$-orbits, so to express the idempotents $e^\cF_P$ in terms of the $\{\alpha_{Q^*}\}$ we only care about the rows of $\Mob$ corresponding to $Q^*\in \cP^*$, and ignore all the other rows.
Consequently, $\FMob$ is the matrix that expresses the idempotents $e^\cF_P$ in terms of the $\{\alpha_{Q^*}\}$.

The inverse $\FMark=\FMob^{-1}$ therefore expresses the irreducible $\cF$-stable sets $\alpha_P$ in terms of the $\cF$-stable idempotents $e^\cF_Q$, which exactly reduces to counting the $Q$-fixed points of $\alpha_P$.
\end{proof}

\begin{example}\label{exMatrixmethod}
Let $S=D_8$ be the dihedral group of order 8 and $\cF=\cF_{D_8}(A_6)$ denote the fusion system induced by the finite group $A_6$. Let $\cP$ be the entire subgroup poset of $D_8$ and $\cP/\cF$ the poset of $\cF$-conjugacy classes of subgroups in $D_8$:
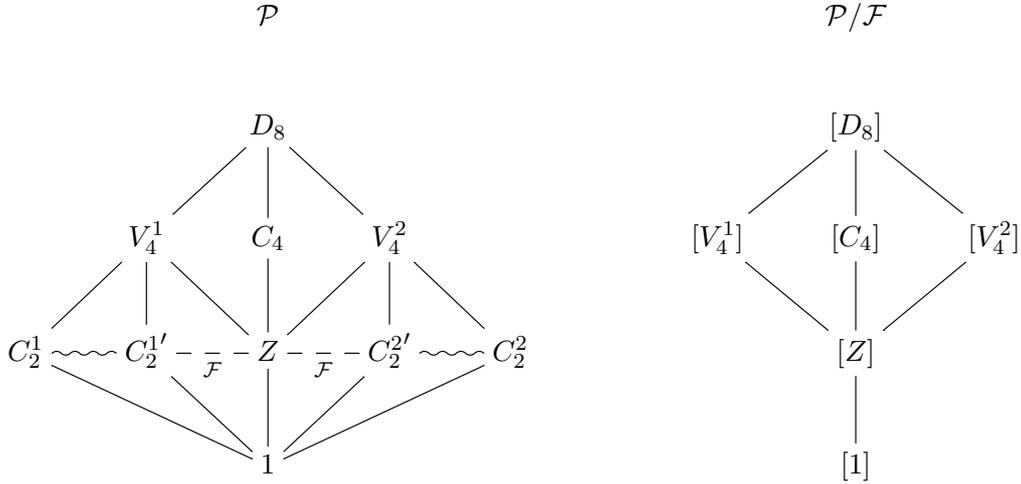
\[
\xymatrix{
&&\cP&&&&&\cP/\cF\\
&&D_8&&&&&[D_8]\\
&V_4^1\ar@{-}[ur]&C_4\ar@{-}[u]&V_4^2\ar@{-}[ul]&&&[V_4^1]\ar@{-}[ur]&[C_4]\ar@{-}[u]&[V_4^2]\ar@{-}[ul]\\
C_2^1\ar@{-}[ur]\ar@{~}[r]&{C_2^1}'\ar@{-}[u]\ar@{--}_-\cF[r]&Z\ar@{-}[ul]\ar@{-}[u]\ar@{-}[ur]&{C_2^2}'\ar@{-}[u]\ar@{~}[r]\ar@{--}^-\cF[l]&C_2^2\ar@{-}[ul]&&&[Z]\ar@{-}[ul]\ar@{-}[u]\ar@{-}[ur]&\\
&&1\ar@{-}[ull]\ar@{-}[ul]\ar@{-}[u]\ar@{-}[ur]\ar@{-}[urr]&&&&&[1]\ar@{-}[u]}
\]
where $V_4^*$ is an elementary abelian 2-group of order 4, $C_n^*$ is a cyclic group of order $n$, $Z\cong C_2$ is the center of $D_8$, and square brackets denote the $\cF$-conjugacy class. The horizontal squiggly lines indicate subgroups' being in the same $D_8$-conjugacy class and dashed lines means that they are in the same $\cF$-conjugacy class.

The table of marks $\Mark$ and its inverse $\Mob=\Mark^{-1}$ are given below
\[
\begin{array}{l|rrrrrrrr}
\Mark &1&C_2^1&Z&C_2^2&V_4^1&C_4&V_4^2&D_8\\
\hline
1	&	8&	4&	4&	4&	2&	2&	2&	1\\
C_2^1&	&	2&	0&	0&	2&	0&	0&	1\\
Z	&	&	&	4&	0&	2&	2&	2&	1\\
C_2^2&	&	&	&	2&	0&	0&	2&	1\\
V_4^1&	&	&	&	&	2&	0&	0&	1\\
C_4	&	&	&	&	&	&	2&	0&	1\\
V_4^2&	&	&	&	&	&	&	2&	1\\
D_8&	&	&	&	&	&	&	&	1\\
\end{array}
\]
\[
\begin{array}{l|rrrrrrrr}
 \Mob &1&C_2^1&Z&C_2^2&V_4^1&C_4&V_4^2&D_8\\
\hline
1 &1/8 & -1/4 & -1/8 & -1/4 & 1/4 & 0 & 1/4& 0\\
C_2^1&	&1/ 2 & 0 & 0 & -1/ 2 & 0 & 0 & 0 \\
Z	&	&	&	1/4 &0 & -1/4 & -1/4& -1/4& 1/2\\
C_2^2&	&	&	&	1/2 & 0	&	0 &	-1/2 &	0\\
V_4^1&	&	&	&	&	1/2 &	0&	0&	-1/2\\
C_4	 &	&	&	&	&	&	1/2 & 0 &	-1/2 \\
V_4^2 &	&	&	&	&	&	&	1/2 &	-1/2\\
D_8 &	&	&	&	&	&	&	&	1 \\\end{array}
\]
Below we give the matrix for $\FMob$ and its inverse $\FMark=\FMob ^{-1}$. Recall that the matrix for $\FMob$ is obtained by adding the columns of $\Mob$ for the subgroups which are $\cF$-conjugate, and then by choosing a fully normalized subgroup in every $\cF$-conjugacy class on the rows. Here $Z$ is the unique fully normalized subgroup in its $\cF$-conjugacy class.
\[
\begin{array}{l|rrrrrr}
\FMob &	1&	Z&	V_4^1&	C_4&	V_4^2&	D_8\\
\hline
1&		1/8&	-5/8&	1/4&	0&	1/4&	0\\
Z&		&	1/4&	-1/4&	-1/4&	-1/4&	1/2\\
V_4^1&	&	&	1/2&	0&0	&	-1/2\\
C_4&	&	&	&	1/2&	0&	-1/2\\
V_4^2&	&	&	&	&	1/2&	-1/2\\
D_8&	&	&	&	&	&	1
\end{array}
\]
\[
\begin{array}{l|rrrrrr}
\FMark&	1&	Z&	V_4^1&	C_4&	V_4^2&	D_8\\
\hline
1&		8&	20&	6&	10&	6&	1\\
Z&		&	4&	2&	2&	2&	1\\
V_4^1&	&	&	2&	0&	0&	1\\
C_4&	&	&	&	2&	0&	1\\
V_4^2&	&	&	&	&	2&	1\\
D_8&	&	&	&	&	&	1
\end{array}
\]

From this we obtain the $\Phi(\alpha_P)$ by reading off the columns of $\cF Mark$ (since $e_Z^\cF=e_{C_2^1}+e_Z+e_{C_2^2}$):
\[
\begin{array}{lllll}
\Phi(\alpha_1)&=&8e^\cF_1&=&8e_1\\
\Phi(\alpha_Z)&=&20e^\cF_1+4e^\cF_Z&=&20e_1+4e_{C_2^1}+4e_Z+4e_{C_2^2}\\
\Phi(\alpha_{V_4^1})&=&6e^\cF_1+2e^\cF_Z+2e^\cF_{V_4^1}&=&6e_1+2e_{C_2^1}+2e_Z+2e_{C_2^2}+2e_{V_4^1}\\
\Phi(\alpha_{C_4})&=&10e^\cF_1+2e^\cF_Z+2e^\cF_{C_4}&=&10e_1+2e_{C_2^1}+2e_Z+2e_{C_2^2}+e_{C_4}\\
\Phi(\alpha_{V_4^2})&=&6e^\cF_1+2e^\cF_Z+2e^\cF_{V_4^2}&=&6e_1+2e_{C_2^1}+2e_Z+2e_{C_2^2}+2e_{V_4^2}\\
\Phi(\alpha_{D_8})&=&e^\cF_1+e^\cF_Z+e^\cF_{V_4^1}&=&e_1+e_{C_2^1}+e_Z+e_{C_2^2}\\
&&\phantom{e^\cF_1}+e^\cF_{C_4}+e^\cF_{V_4^2}+e^\cF_{D_8}&&\phantom{e_1}+e_{V_4^1}+e_{C_4}+e_{V_4^2}+e_{D_8}
\end{array}
\]
Finally, applying the matrix \emph{M\"ob} to each of these fixed point vectors yields the $S$-orbit description of the $\alpha_{P}$:
\[
\begin{array}{lll}
\alpha_1&=&[S/1]\\
\alpha_Z&=&[S/Z]+2[S/C_2^1]+2[S/C_2^2]\\
\alpha_{V_4^1}&=&[S/V_4^1]+[S/C_2^2]\\
\alpha_{C_4}&=&[S/C_4]+[S/C_2^1]+[S/C_2^2]\\
\alpha_{V_4^2}&=&[S/V_4^2]+[S/C_2^1]\\
\alpha_{D_8}&=&[S/D_8]
\end{array}
\]
\end{example}

There is an explicit  formula for expressing the idempotent basis $\{ e_Q\}$ in terms of the transitive $S$-set basis $\{[S/P]\}$ using the combinatorics of the subgroup poset, which is often referred as the Gluck's idempotent formula \cite{Gluck}. In  the following two sections we find similar explicit formulas for the coefficients of $\alpha_{P^*}$ with respect to the idempotent basis $\{e_Q\}$ and then with respect to the $S$-set basis $\{[S/P]\}$. For this we need to look at the M{\" o}bius inversion in Gluck's idempotent formula more closely.

\section{Fixed point orders and M\"obius inversion}\label{secMobiusInversion}

In this section we discuss how a more explicit formula can be obtained for fixed point orders of basis elements using M\" obius inversion. We first introduce basic definitions about M\" obius inversion. For more details, we refer the reader to \cite{RotaFoundations}.

Let $\cP$ be a finite poset. The \emph{incidence function of $\cP$} is defined as the function
\[
\zeta_\cP:\cP\times\cP\to\ZZ:(a,b)\mapsto\begin{cases}1 &a\leq b,\\ 0&\textrm{else.}\end{cases}
\]
The \emph{incidence matrix} of $\cP$ is  the $\abs\cP\times\abs\cP$-matrix $(\zeta_\cP)$ with entries  $(\zeta_\cP)_{a,b}=\zeta_\cP(a,b)$. When labelling the rows/columns we respect the partial order of $\cP$, such that $a\leq b$ in $\cP$ implies that the $a$-row/-column precedes the $b$-row/-column. This way the incidence matrix is always upper unitriangular (an upper triangular matrix with all diagonal entries equal to $1$).

\begin{definition}
The \emph{M\"obius function} for a poset $\cP$ is $\mu_\cP:\cP\times\cP\to\QQ$ defined by
\[
\sum_{a\in\cP}\zeta_\cP(x,a)\mu_\cP(a,y)=\delta_{x,y}=\sum_{a\in\cP}\mu_\cP(x,a)\zeta_\cP(a,y)
\]
for all $x,y\in\cP$.  If the corresponding $\abs\cP\times\abs\cP$ \emph{M\"obius matrix} is  $(\mu_\cP)$, we have
$
(\mu_\cP)=(\zeta_\cP)^{-1}.
$
\end{definition}

\begin{lemma}\label{lemMobiusIntegral}
$\mu_\cP(a,b)\in\ZZ$ for all $a,b\in\cP$.
\begin{proof}
By our labelling convention, $(\zeta_\cP)$ is upper unitriangular.  Therefore we can write $(\zeta_\cP)=I+(\eta_\cP)$, where $(\eta_\cP)_{i,j}=1$ when $a_i<a_j$ and vanishes elsewhere.  Then $\eta_\cP$ is strictly upper triangular, and $(\eta_\cP)^{\abs\cP}=0$, so
\[
(\mu_\cP)=(\zeta_\cP)^{-1}=(I+(\eta_\cP))^{-1}=I-(\eta_\cP)+(\eta_\cP)^2-(\eta_\cP)^3+-\ldots+(-1)^{\abs\cP-1}(\eta_\cP)^{\abs\cP-1}
\]
has all integral entries.
\end{proof}
\end{lemma}

Each of the matrices $(\eta^k_\cP):=(\eta_\cP)^k$ has an interpretation in terms of chains in the poset $\cP$.

\begin{definition}
A \emph{chain of length $k$ in $\cP$} is a totally ordered subset of $k+1$ elements
$
\sigma=\{a_0<a_1<\ldots<a_k\}.
$
Such a chain \emph{links $a_0$ to $a_k$}.

Let $\fC^k_\cP(a,b)$ be the set of chains of length $k$ linking $a$ to $b$, and $\fC^k_\cP$ the set of all chains of length $k$ in $\cP$.  $\fC_\cP^0$ is the set of elements of $\cP$. Similarly, let $\fC_\cP(a,b)$ be the set of all chains linking $a$ to $b$, $\fC_\cP$ the set of all chains in $\cP$, and for any chain $\sigma\in \fC_\cP$ let $\abs{\sigma}$ denote the length of $\sigma$.
\end{definition}

\begin{lemma}\label{lemCountingChains}
$(\eta^k_\cP)_{a,b}=|\fC^k_\cP(a,b)|$.
\end{lemma}

\begin{proof}
With $a_0:=a$ and $a_k:=b$, the definition of matrix multiplication gives us
\[(\eta_\cP^k)_{a_0,a_k} = \sum_{a_1,\dotsc,a_{k-1}\in \cP} \eta_\cP(a_0,a_1)\cdot \eta_\cP(a_1,a_2) \dotsm \eta_\cP(a_{k-1},a_k).\]
By definition of the incidence function $\eta_\cP$, each factor $\eta_\cP(a_i,a_{i+1})$ is $1$ if $a_i< a_{i+1}$ and zero otherwise. The product $\eta_\cP(a_0,a_1) \dotsm \eta_\cP(a_{k-1},a_k)$ is therefore nonzero and equal to~$1$ precisely when $a_0<a_1<\dotsb<a_k$ is a $k$-chain in $\cP$ linking $a_0$ to $a_k$.
\end{proof}

\begin{prop}
For all $a,b\in \cP$,
\[\mu_\cP(a,b) = \sum_{k=0}^\infty (-1)^k\abs{\fC_\cP^k(a,b)} = \sum_{\sigma\in \fC_\cP(a,b)}(-1)^{\abs\sigma}.\]
\end{prop}

\begin{proof}
Immediate from Lemmas \ref{lemMobiusIntegral} and \ref{lemCountingChains} and their proofs.\end{proof}

\begin{remark} From the formula above it is clear that  the M{\" o}bius function can be expressed as the reduced Euler characteristic of a subposet in $\cP$. More specifically, for $a<b$, let $(a, b)_{\cP}$ denote the poset of all $c\in \cP$ with $a< c < b$. Then $\mu _{\cP} (a,b)$ is equal to the reduced Euler characteristic $\widetilde \chi ((a,b)_{\cP} )$ of the subposet $(a,b)_{\cP}$ for every $a, b\in \cP$ such that $a<b$.
\end{remark}

\subsection{M\"obius functions and fixed points}
On the next few pages we go through the construction of the matrix $\FMark$ in Theorem \ref{thmMatrixForAlphas} again, but this time we follow the calculations in detail using the framework of incidence and M\"obius functions.
For a finite $p$-group $S$, we let $\cP$ be the poset of subgroups ordered by inclusion. This poset has incidence and M\"obius functions $\zeta$ and $\mu$ as described in the previous section.

The matrix of marks $\Mark$ for the Burnside ring of $S$ has entries $\Mark_{[Q],[P]} = \abs{N_S(Q,P)}/\abs{P}$ defined for pairs $[Q]_S$,
 $[P]_S$ of $S$-conjugacy classes of subgroups in $S$. Each column is divisible by the diagonal entry, which is the order of the Weyl group $W_S P=N_S P/P$.  If we divide the $[P]_S$-column by $\abs{W_SP}$, we get
\begin{align*}
\Mark_{[Q],[P]}\cdot \frac 1{\abs {W_S P}} &= \frac{\abs{N_S(Q,P)}}{\abs{N_S P}}= \frac{\abs{\{s\in S \mid \lc sQ \leq P\}}}{\abs{N_S P}} = \frac{\abs{\{s\in S \mid Q \leq P^s\}}}{\abs{N_S P}}\\&= \abs{\{P'\leq S \mid P'\sim_S P\text{ and } Q\leq P'\}}
 = \sum_{P'\sim_S P} \zeta(Q,P').
\end{align*}
We denote this value by $\tilde \zeta_S([Q],[P])$, and we call $\tilde \zeta_S$ the \emph{modified incidence function} for the $S$-conjugacy classes of subgroups. We have $(\tilde\zeta_S)_{[Q],[P]}= \Mark_{[Q],[P]} / \abs{W_S P}$, so the modified incidence matrix $(\tilde\zeta_S)$ is upper unitriangular (see Example \ref{exMobius}).

Inverting the matrix $(\tilde\zeta_S)$, we define $(\tilde\mu_S):=(\tilde\zeta_S)^{-1}$ which gives rise to a \emph{modified M\"obius function} $\tilde\mu_S$ for $S$-conjugacy classes of subgroups. Since $\Mob=\Mark^{-1}$ is the inverse of the matrix of marks, we have $(\tilde\mu_S)_{[Q],[P]} = \abs{W_S Q}\cdot \Mob_{Q,P}$. As $(\tilde\zeta_S)$ is triangular with diagonal entries $1$, we also have $(\tilde\mu_S) = (\tilde\zeta_S)^{-1} = \sum_{k=0}^\infty (-1)^k\cdot((\tilde\zeta_S) - I)^k$ as in the proof of Lemma \ref{lemMobiusIntegral}, which we use to calculate the entries of $(\tilde\mu_S)$:
\begin{align*}
\tilde\mu_S  ([Q],[P])
=& \ \sum_{k=0}^\infty (-1)^k\cdot((\tilde\zeta_S) - I)^k \\ =&\sum_{([R_0],[R_1],\dotsc ,[R_{k}])\in \cT_S}  (-1)^k \tilde\zeta_S([R_0],[R_1])\dotsm\tilde\zeta_S([R_{k-1}],[R_k])
\end{align*}
where $\cT_S$ consists of all tuples $([R_0], [R_1], ..., [R_k])$, for $k\geq 0$, of $S$-conjugacy classes of subgroups $[R_i]\in Cl(S)$ such that $[R_0]=[Q]$, $[R_k]=[P]$, and $\abs {R_0}< \abs{R_1} < \dotsb < \abs {R_k}.$
Since we have
\[\tilde\zeta_S([R_i],[R_j])= \sum_{R'_j\sim_S R_j} \zeta(R_i  ,R'_j) \] for all $i,j$, we obtain that $\tilde\mu_S  ([Q],[P])$ is equal to the sum

\begin{align*}
\sum_{\substack{([R_0],[R_1],\dotsc ,[R_{k}])\in \cT_S \\ R_0=Q}}& \ \sum_{R_1'\sim_S R_1} (-1)^k\zeta(R_0,R_1')\tilde\zeta_S([R_1'],[R_2])\dotsm\tilde\zeta_S([R_{k-1}],[R_k])
\\ \vdots\ \
\\ =\ \sum_{\substack{([R_0],[R_1],\dotsc,[R_{k}])\in \cT_S \\ R_0=Q} }  & \ \sum_{R_1'\sim_S R_1}\sum_{R_2'\sim_S R_2} \dotsi \sum_{R_k'\sim_S R_k} (-1)^k \zeta(R_0,R_1')\zeta(R_1',R_2')\dotsm\zeta(R_{k-1}',R_k') \\ =\ \sum_{\substack{R_0< R_1' < \dotsb < R_k'\\ \text{s.t. } R_0=Q,\ R_k'\sim_S P}} & (-1)^k
 = \sum_{P'\sim_S P}\ \sum_{k=0}^\infty (-1)^k \abs{\fC_\cP^k(Q,P')}
 = \sum_{P'\sim_S P} \mu(Q,P').
\end{align*}

Therefore, the matrix $\Mob$, the inverse of the matrix of marks, has entries
\[\Mob_{[Q],[P]} = \frac1{\abs{W_S Q}} \tilde\mu_S([Q],[P]) = \frac1{\abs{W_S Q}} \sum_{P'\sim_S P} \mu(Q,P').\]
This concludes the part of our investigation concerning only the subgroup structure of $S$, and for the calculations below we include the extra data of a saturated fusion system $\cF$ on $S$.

In order to determine the number of fixed points $\abs{(\alpha_P)^Q}$ as in Theorem \ref{thmMatrixForAlphas}, we wish to calculate the $\cF$-analogs of $\Mark$ and $\Mob$ above. To do this, we first choose a fully normalized representative $P^*$ for each $\cF$-conjugacy class $[P]_\cF$ of subgroups, and as before let $\cP^*$ be the collection of these representatives.
Recall that the matrix $\FMob$ is constructed from $\Mob$ by picking out the rows corresponding to $Q^*\in \cP^*$, and the column in $\FMob$ corresponding to $P^*\in \cP^*$ is the sum of the columns in $\Mob$ corresponding to $[P]_S$ with $P\sim_\cF P^*$. More explicitly, we have
\[\FMob_{Q^*,P^*} := \sum_{[P]_S\subseteq [P^*]_\cF} \Mob_{[Q^*],[P]} = \frac1{\abs{W_S Q^*}}\sum_{P\sim_\cF P^*} \mu(Q^*,P).\]
We define the \emph{modified M\"obius function} $\tilde\mu_\cF\colon \cP^*\x \cP^* \to \ZZ$ for the (representatives of) $\cF$-conjugacy classes of subgroups, to be
\[\tilde\mu_\cF(Q^*,P^*) := \abs{W_S Q^*}\cdot \FMob_{Q^*,P^*} = \sum_{P\sim_\cF P^*} \mu(Q^*,P),\]
summing the usual M\"obius function. The associated matrix $(\tilde\mu_\cF)$ is then upper unitriangular.

The modified incidence matrix for $\cF$ is defined as the inverse $(\tilde\zeta_\cF):=(\tilde\mu_\cF)^{-1}$, with the associated function $\tilde\zeta_\cF\colon \cP^*\x \cP^*\to \ZZ$.
By Theorem \ref{thmMatrixForAlphas} we then have
\[\abs{(\alpha_P)^Q} = \FMark_{Q^*,P^*} = \abs{W_S P^*}\cdot \tilde\zeta_\cF(Q^*,P^*)\]
where $\FMark:= \FMob^{-1}$. Recall that for each subgroup $R\leq S$, we denote by $R^*$ the chosen fully normalized representative for the $\cF$-conjugacy class of $R$. As previously, the fact that $(\tilde\mu_\cF)$ is unitriangular implies
\begin{align*}
\tilde\zeta_\cF(Q^*,P^*) =& \sum_{k=0}^\infty (-1)^k\cdot((\tilde\mu_\cF) - I)^k \notag
\\ =\ & \sum_{(R_0^*,R_1^*,\dotsc, R_k^*)\in \cT_{\cF}}   (-1)^k\tilde\mu_\cF(R_0^*,R_1^*)\dotsm\tilde\mu_\cF(R_{k-1}^*,R_k^*) \notag
\end{align*}
where $\cT_\cF$ consists of all tuples $(R_0^*,R_1^*,\dotsc, R_k^*)$, for all $k\geq 0$, of $\cF$-conjugacy class representatives $R_i^* \in \cP^*$ such that $R_0^*=Q^*$, $R_k^*=P^*$, and $\abs{R_0^*}<\abs{R_1^*}<\dotsb <\abs{R_k^*}$. Since we have \[\tilde\mu_\cF(R_i^*,R_j^*) = \sum_{R_j\sim_\cF R_j^*} \mu(R_i^*,R_j),\]
for all $R_i^*,R_j^*\in \cP^*$, we obtain that

\begin{align}
\tilde\zeta_\cF(Q^*,P^*)=\sum_{(R_0^*,R_1^*,\dotsc, R_k^*)\in \cT_{\cF} } &\ \sum_{R_1\sim_\cF R_1^*} (-1)^k\mu(R_0^*,R_1)\tilde\mu_\cF(R_1^*,R_2^*)\dotsm\tilde\mu_\cF(R_{k-1}^*,R_k^*) \notag
\\\vdots\ \  \notag
\\ =\  \sum_{\substack{(R_0^*, R_1 ^*, \dotsc, R_k ^* )\in \cT_{\cF}, \\  R_1,\dotsc ,R_k\in \cP  \text{ s.t. } R_i\sim_\cF R_i^*}} & \ (-1)^k\mu(R_0^*,R_1)\mu(R_1^*,R_2)\dotsm\mu(R_{k-1}^*,R_k) \label{eqFixedPointMobiusChain}
\\ =\ \sum_{\substack{(R_0^*, R_1 ^*, \dotsc, R_k ^* )\in \cT_{\cF}, \\  R_1,\dotsc ,R_k\in \cP  \text{ s.t. } R_i\sim_\cF R_i^*}}
& \ \sum_{\substack{\sigma_i\in \fC_\cP(R_{i-1}^*,R_i)\\\text{for $1\leq i\leq k$}}}(-1)^{k+\abs{\sigma_1}+\dotsb+\abs{\sigma_k}}. \notag
\end{align}

To calculate $\tilde\zeta_\cF(Q^*,P^*)$ we hence have to count sequences of chains $(\sigma_1,\dotsc,\sigma_k)$ such that the end $R_i$ of $\sigma_i$ is $\cF$-conjugate to the start $R_i^*$ of $\sigma_{i+1}$, and the first chain $\sigma_1$ has to start at $Q^*$ while the final chain $\sigma_k$ only has to end at $P^*$ up to $\cF$-conjugation. We give these sequences a name:

\begin{definition}\label{defTetheredBrokenChain}
A \emph{tethered $\cF$-broken chain in $\cP$ linking $Q^*\in \cP^*$ to $P\in \cP$} is a sequence of chains $(\sigma_1,\dotsc,\sigma_k)$ in $\cP$ subject to the following requirements. With each chain written as $\sigma_i=(a^i_0, \dotsc,a^i_{n_i})$ they must satisfy
\begin{itemize}
\item $a^i_{n_i}\sim_\cF a^{i+1}_0$ for all $1\leq i\leq k-1$, so the endpoints of the chains fit together up to conjugation in $\cF$.
\item $a^i_0\in \cP^*$ for all $1\leq i\leq k$. Every chain starts at one of the chosen representatives.
\item $\abs{\sigma_i}=n_i>0$, for all $1\leq i\leq k$.
\item $a^1_0 = Q^*$ and $a^k_{n_k}\sim_\cF P$.
\end{itemize}
If $Q^*\sim_\cF P$, we allow the trivial broken chain with $k=0$.
Let $\tBC_\cF (Q^*,P)$ be the set of tethered $\cF$-broken chains linking $Q^*$ to $P$. The \emph{total length} of a tethered broken chain $\sigma=(\sigma_1,\dotsc,\sigma_k)$ is defined  to be
\[\ell(\sigma) := k+\abs{\sigma_1}+\dotsb+\abs{\sigma_k}.\]
\end{definition}

We visualize a tethered broken chain as a zigzag diagram in the following way:
\[
\begin{tikzpicture}
\matrix (M) [matrix of math nodes, column sep=.5cm, row sep=.5cm] {
a^1_0 & \dotsb & a^1_{n_1} &&&[-.8cm] &[-.8cm] && \\
&& a^2_0 &\dotsb &a^2_{n_2} &&&& \\[-.3cm]
&&&&\phantom{a^3_0}& \ddots &\phantom{a^p_q}&&\\[-.3cm]
&&&&&& a^k_0 & \dotsb & a^k_{n_k} \\
};
\path[sloped]
(M-1-1) -- node{$<$} (M-1-2) -- node{$<$} (M-1-3) -- node{$\sim$} (M-2-3) -- node{$<$} (M-2-4) -- node{$<$} (M-2-5) -- node{$\sim$} (M-3-5)
(M-3-7) -- node{$\sim$} (M-4-7) -- node{$<$} (M-4-8) -- node{$<$} (M-4-9)
;
\end{tikzpicture}
\]
The total length of the tethered broken chain is then the total number of $<$ and $\sim$ signs plus $1$. The added $1$ can be viewed as an additional hidden $Q\sim_{\cF} Q^*$ in front of the broken chain, and this interpretation matches the description, in Remark \ref{rmkTetheredIsASpecialCase} below, of tethered broken chains as a special case of the \emph{broken chains} defined in Section \ref{secBrokenChains}.

With the terminology of tethered broken chains, the calculations above translate to the following statements:
\begin{prop}
The modified incidence function $\tilde\zeta_\cF$ for a saturated fusion system $\cF$, can be calculated as
\[\tilde\zeta_\cF(Q^*,P^*) = \sum_{\sigma\in \tBC_\cF(Q^*,P^*)} (-1)^{\ell(\sigma)} =\sum_{(\sigma_1,\dotsc,\sigma_k)\in \tBC_\cF(Q^*,P^*)} (-1)^{k+\abs{\sigma_1}+\dotsb+\abs{\sigma_k}}\]
for all fully normalized representatives $Q^*, P^*\in \cP^*$.
\end{prop}

We now state the main result of this section.

\begin{theorem}\label{thmFixedPointBrokenChains}
Let $\cF$ be a saturated fusion system over a finite $p$-group $S$, and let $\cP^*$ be a set of fully normalized representatives for the $\cF$-conjugacy classes of subgroups in $S$. Let  $\tBC_\cF(Q^*,P^*)$ denote the set of all tethered $\cF$-broken chains linking $Q^*$ to $P^*$.
Then the numbers of fixed points for the irreducible $\cF$-stable sets $\alpha_{P^*}$, $P^*\in \cP^*$, can be calculated as
\[\abs*{(\alpha_{P^*})^{Q^*}} = \abs{W_S P^*}\cdot \sum_{\sigma\in \tBC_\cF(Q^*,P^*)} (-1)^{\ell(\sigma)}\]
for $Q^*,P^*\in \cP^*$.
\end{theorem}

\begin{proof}
Immediate from the proposition since $\abs{(\alpha_{P^*})^{Q_*}} = \abs{W_S P^*}\cdot \tilde\zeta_\cF(Q^*,P^*)$.
\end{proof}

\begin{example}\label{exMobius}

Let $S=D_8$ and $\cF=\cF _{S}(A_6)$ as before. The incidence matrix $\zeta_\cP$ and the M\" obius matrix $\mu _\cP$ are given as follows.
\[
\begin{array}{l|rrrrrrrrrr}
\zeta_\cP&1&C_2^1&{C_2^1}'&Z&{C_2^2}'&C_2^2&V_4^1&C_4&V_4^2&D_8\\
\hline
1	&	1&	1&	1&	1&	1&	1&	1&	1&	1&	1\\
C_2^1&	&	1&	&	&	&	&	1&	&	&	1\\
{C_2^1}'&	&	&	1&	&	&	&	1&	&	&	1\\
Z	&	&	&	&	1&	&	&	1&	1&	1&	1\\
{C_2^2}'&	&	&	&	&	1&	&	&	&	1&	1\\
C_2^2&	&	&	&	&	&	1&	&	&	1&	1\\
V_4^1&	&	&	&	&	&	&	1&	&	&	1\\
C_4	&	&	&	&	&	&	&	&	1&	&	1\\
V_4^2&	&	&	&	&	&	&	&	&1	&	1\\
D_8&	&	&	&	&	&	&	&	&	&	1
\end{array}
\]
\[
\begin{array}{l|r:rr:r:rr:r:r:r:r}
\mu_\cP&1&C_2^1&{C_2^1}'&Z&{C_2^2}'&C_2^2&V_4^1&C_4&V_4^2&D_8\\
\hline
\underline{1}	&	1&	-1&	-1&	-1&	-1&	-1&	2&	0&	2&	0\\
\hdashline
\underline{C}_2^1&	&	1&	&	&	&	&	-1&	&	&	0\\
{C_2^1}'&	&	&	1&	&	&	&	-1&	&	&	0\\
\hdashline
\underline{Z}	&	&	&	&	1&	&	&	-1&	-1&	-1&	2\\
\hdashline
{C_2^2}'&	&	&	&	&	1&	&	&	&	-1&	0\\
\underline{C}_2^2&	&	&	&	&	&	1&	&	&	-1&	0\\
\hdashline
\underline{V}_4^1&	&	&	&	&	&	&	1&	&	&	-1\\
\hdashline
\underline{C}_4	&	&	&	&	&	&	&	&	1&	&	-1\\
\hdashline
\underline{V}_4^2&	&	&	&	&	&	&	&	&1	&	-1\\
\hdashline
\underline{D}_8&	&	&	&	&	&	&	&	&	&	1
\end{array}
\]
Below we see the matrices for $\widetilde \mu _S$ and $\widetilde \zeta_S$ obtained by summing over the columns of subgroups belonging to the same $S$-conjugacy class and choosing an $S$-conjugacy class representative on the rows.
\[
\begin{array}{l|rrrrrrrr}
\widetilde\mu_S&1&C_2^1&Z&C_2^2&V_4^1&C_4&V_4^2&D_8\\
\hline
1	&	1&	-2&	-1&	-2&	2&	0&	2&	0\\
C_2^1&	&	1&	0&0	&	-1&0	&0	&	0\\
Z	&	&	&	1&0	&	-1&	-1&	-1&	2\\
C_2^2&	&	&	&	1&0	&0	&	-1&	0\\
V_4^1&	&	&	&	&	1&0	&0	&	-1\\
C_4	&	&	&	&	&	&	1&0	&	-1\\
V_4^2&	&	&	&	&	&	&1	&	-1\\
D_8&	&	&	&	&	&	&	&	1
\end{array}
\qquad
\begin{array}{l|rrrrrrrr}
\widetilde\zeta_S&1&C_2^1&Z&C_2^2&V_4^1&C_4&V_4^2&D_8\\
\hline
1	&	1&	2&	1&	2&	1&	1&	1&	1\\
C_2^1&	&	1&	0&	0&	1&	0&	0&	1\\
Z	&	&	&	1&	0&	1&	1&	1&	1\\
C_2^2&	&	&	&	1&	0&	0&	1&	1\\
V_4^1&	&	&	&	&	1&	0&	0&	1\\
C_4	&	&	&	&	&	&	1&	0&	1\\
V_4^2&	&	&	&	&	&	&1	&	1\\
D_8&	&	&	&	&	&	&	&	1
\end{array}
\]
\[
\begin{array}{l|rrrrrrrr}
\mathbf{W}_{D_8}&1&C_2^1&Z&C_2^2&V_4^1&C_4&V_4^2&D_8\\
\hline
1	&	8&	&	&	&	&	&	&	\\
C_2^1&	&	2&	&	&	&	&	&	\\
Z	&	&	&	4&	&	&	&	&	\\
C_2^2&	&	&	&	2&	&	&	&	\\
V_4^1&	&	&	&	&	2&	&	&	\\
C_4	&	&	&	&	&	&	2&	&	\\
V_4^2&	&	&	&	&	&	&	2&	\\
D_8&	&	&	&	&	&	&	&	1
\end{array}
\qquad
\begin{array}{l|rrrrrrrr}
\mathbf{m}_{D_8}&1&C_2^1&Z&C_2^2&V_4^1&C_4&V_4^2&D_8\\
\hline
1	&	8&	4&	4&	4&	2&	2&	2&	1\\
C_2^1&	&	2&	0&	0&	2&	0&	0&	1\\
Z	&	&	&	4&	0&	2&	2&	2&	1\\
C_2^2&	&	&	&	2&	0&	0&	2&	1\\
V_4^1&	&	&	&	&	2&	0&	0&	1\\
C_4	&	&	&	&	&	&	2&	0&	1\\
V_4^2&	&	&	&	&	&	&	2&	1\\
D_8&	&	&	&	&	&	&	&	1
\end{array}
\]
The last two matrices above are the diagonal matrix $\mathbf{W}_{D_8}$ with entries $(\mathbf {W}_{D_8})_{[P],[P]}=\abs {W_S (P)}$, and  the matrix $\mathbf{m}_{D_8}= \widetilde \zeta _S \cdot \mathbf{W} _{D_8}$ which  is the same as matrix of the mark homomorphism $\Mark$. So we also have $\Mob =\mathbf{W}_{D_8} ^{-1}\cdot \widetilde \mu_S$.

Now let $\widetilde \mu _{\cF}$ be the matrix obtained by summing columns of $\widetilde \mu _S$ over the $\cF$ conjugacy classes and picking fully normalized representatives for the rows. Let $\widetilde \zeta _{\cF}=(\widetilde \mu _\cF)^{-1}$.
\[
\begin{array}{l|rrrrrr}
\widetilde\mu_\cF&	1&	Z&	V_4^1&	C_4&	V_4^2&D_8\\\hline
1&		1&	-5&	2&	0&	2&	0\\
Z&		&	1&	-1&	-1&	-1&	2\\
V_4^1&	&	&	1&0	&0	&	-1\\
C_4&	&	&	&	1&0	&	-1\\
V_4^2&	&	&	&	&	1&	-1\\
D_8&	&	&	&	&	&	1
\end{array}
\qquad
\begin{array}{l|rrrrrr}
\widetilde\zeta_\cF&	1&	Z&	V_4^1&	C_4&	V_4^2&	D_8\\
\hline
1&		1&	5&	3&	5&	3&	1\\
Z&		&	1&	1&	1&	1&	1\\
V_4^1&	&	&	1&	0&	0&	1\\
C_4&	&	&	&	1&	0&	1\\
V_4^2&	&	&	&	&	1&	1\\
D_8&	&	&	&	&	&	1
\end{array}
\]
From the definition of $\FMob$, it is easy to see that $\FMob =\mathbf{W}_\cF ^{-1}\cdot \widetilde \mu _\cF $  and $\FMark=\FMob ^{-1}= \widetilde \zeta _\cF \cdot \mathbf{W} _\cF$ where
$\mathbf{W}_\cF$ is the diagonal matrix with entries $({\mathbf W}_\cF)_{P^*,P^*}=\abs {W_S (P^*)}$ for all $P^*\in \cP^*$. Theorem \ref{thmFixedPointBrokenChains} says that we can calculate the entries of the matrix $\widetilde \zeta _\cF$ by counting the number of tethered broken chains. For example, $\widetilde \zeta _\cF (1, Z)=5$ because there are $5$ tethered broken chains linking $1$ to $Z$. We give more complicated examples of tethered broken chain calculations in Example \ref{exBrokenChains}.
\end{example}

\begin{remark} Note that the modified incidence matrix with respect to $S$-conjugations and the modified M{\" o}bius function on $S$-conjugate subgroups (coming from the poset of subgroups) are constructed in the same way:  Add the columns of $S$-conjugate subgroups, pick out any row from each class. It is interesting that performing the same operation on the originals of the incidence function and the M{\" o}bius inverse ends up giving you inverse matrices; in particular, this is not what happens for modifications with respect to $\cF$-conjugation which is what is done in the rest of the paper.  We think that this shows that the
$S$-conjugation action on the subgroup poset is more special that the $\cF$-conjugation action.
\end{remark}

\section{Broken chains and the main theorem}\label{secBrokenChains}

Now that we have formulas for the number of fixed points of $\alpha_P$, we will determine how each $\alpha_P$ decomposes into $S$-orbits. For every element $X\in A(S)$ of the Burnside ring, we let $c_Q(X)$ denote the number of (virtual) $[S/Q]$-orbits, i.e. the coefficients of the linear combination $X=\sum_{[Q]_S} c_Q(X)\cdot [S/Q]$. The matrix of marks $\Mark$ encodes the number of fixed points in terms of the number of orbits, so the numbers $\abs{X^Q}$ form a fixed point vector $\ph:= \Mark\cdot(c_Q(X))$. Recall that $\Mob$ is the inverse of $\Mark$. Given any fixed point vector $\ph$, we can therefore recover the orbit decomposition as $(c_Q(X)) = \Mob\cdot \ph$.

For $\alpha _P$ we already have a formula for the number of fixed points $\abs{(\alpha_P)^Q}$, which we write in the form of
\[\abs{(\alpha_P)^Q} = \FMark_{Q^*,P^*} = \abs{W_S P^*}\cdot \tilde\zeta_\cF(Q^*,P^*)\]
where $\tilde\zeta_\cF(Q^*,P^*)$ has a complicated M\" obius formula given in \eqref{eqFixedPointMobiusChain}. We also know how $\Mob$ is given in terms of M\"obius functions. The number of $[S/Q]$-orbits in $\alpha_P$ must therefore be

\begin{align*}
& c_Q(\alpha_P) = \sum_{[R] \in Cl(S)} \Mob_{Q,R} \cdot \abs*{(\alpha_P)^R}
= \frac{1}{\abs{W_S Q}} \sum_{[R] \in Cl(S)} \tilde\mu_S([Q],[R]) \cdot \abs*{(\alpha_{P^*})^{R^*}}
\\ & =\  \frac{1}{\abs{W_S Q}} \sum_{R\in \cP} \mu(Q,R) \cdot \Big(\abs{W_S P^*}\cdot \tilde\zeta_\cF(R^*,P^*)\Big)
\\ & =\ \frac{\abs{W_S P^*}}{\abs{W_S Q}} \sum_{R\in\cP} \mu(Q,R)\cdot\hspace{-.5cm} \sum_{\substack{(R_0^*, R_1 ^*, \dotsc, R_k ^* )\in \cT_{\cF}, \\  R_1,\dotsc ,R_k\in \cP  \text{ s.t. } R_i\sim_\cF R_i^*}} (-1)^k\mu(R_0^*,R_1)\mu(R_1^*,R_2)\dotsm\mu(R_{k-1}^*,R_k)
\end{align*}
where the sum is over $\cT_{\cF}$ of all $k$-tuples, for all $k\geq 0$, of (prefixed) $\cF$-conjugacy  class representatives $R_i^* \in \cP^*$ such that $R_0^*=R^*$, $R_k^*=P^*$, and $\abs{R_0^*}<\abs{R_1^*}<\dotsb <\abs{R_k^*}$. From this we obtain that
\begin{align*}
c_Q(\alpha_P)=\ & \frac{\abs{W_S P^*}}{\abs{W_S Q}} \sum_{\substack{R_0, R_1,\dotsc,R_k\in \cP\\ \text{s.t. } R_k\sim_\cF P^*,\\ \abs{Q}\leq\abs{R_0}<\abs{R_1}<\dotsb <\abs{R_k}}} (-1)^k\mu(Q,R_0)\mu(R_0^*,R_1)\mu(R_1^*,R_2)\dotsm\mu(R_{k-1}^*,R_k)
\\ =\ & \frac{\abs{W_S P^*}}{\abs{W_S Q}} \sum_{\substack{R_0, R_1,\dotsc,R_k\in \cP\\ \text{s.t. } R_k\sim_\cF P^*,\\ \abs{Q}\leq\abs{R_0}<\abs{R_1}<\dotsb <\abs{R_k}}}
\sum_{\sigma_0\in \fC_\cP(Q,R_0)}\
\sum_{\substack{\sigma_i\in \fC_\cP(R_{i-1}^*,R_i)\\\text{for $1\leq i\leq k$}}} (-1)^{k+\abs{\sigma_0}+\abs{\sigma_1}+\dotsb+\abs{\sigma_k}}.
\end{align*}

The resulting formula is very similar to the calculations for fixed points in the previous section, except that we have an additional (possibly trivial) chain $\sigma_0$ in front. We combine this additional chain with the definition of tethered broken chains and arrive at the following definition:

\begin{definition}\label{defBrokenChain}
An \emph{$\cF$-broken chain in $\cP$ linking $Q\in \cP$ to $P\in \cP$} is a sequence of chains $(\sigma_0,\sigma_1,\dotsc,\sigma_k)$ in $\cP$ subject to the following requirements. With each chain written as $\sigma_i=(a^i_0, \dotsc,a^i_{n_i})$ they must satisfy
\begin{itemize}
\item $a^i_{n_i}\sim_\cF a^{i+1}_0$ for all $0\leq i\leq k-1$, so the endpoints of the chains fit together up to conjugation in $\cF$.
\item $a^i_0\in \cP^*$ for all $1\leq i\leq k$. Every chain except for $\sigma_0$ starts at one of the chosen representatives.
\item $\abs{\sigma_i}=n_i>0$, for all $1\leq i\leq k$. Note that $\sigma_0$ is allowed to be trivial.
\item $a^0_0 = Q$ and $a^k_{n_k}\sim_\cF P$.
\end{itemize}
As before, if $Q\sim_\cF P$, we allow the trivial broken chain with $k=0$ and $\sigma_0$ trivial.
Let $\BC_\cF (Q,P)$ be the set of $\cF$-broken chains linking $Q$ to $P$. We define the \emph{total length} of a broken chain $\sigma=(\sigma_0,\dotsc,\sigma_k)$ to be
\[\ell(\sigma) := k+\abs{\sigma_0}+\dotsb+\abs{\sigma_k}.\]
To visualize a broken chain, we represent it by the diagram
\[
\begin{tikzpicture}
\matrix (M) [matrix of math nodes, column sep=.5cm, row sep=.5cm] {
a^0_0 & \dotsb & a^0_{n_0} &&&[-.8cm] &[-.8cm] && \\
&& a^1_0 &\dotsb &a^1_{n_1} &&&& \\[-.3cm]
&&&&\phantom{a^3_0}& \ddots &\phantom{a^p_q}&&\\[-.3cm]
&&&&&& a^k_0 & \dotsb & a^k_{n_k} \\
};
\path[sloped]
(M-1-1) -- node{$<$} (M-1-2) -- node{$<$} (M-1-3) -- node{$\sim$} (M-2-3) -- node{$<$} (M-2-4) -- node{$<$} (M-2-5) -- node{$\sim$} (M-3-5)
(M-3-7) -- node{$\sim$} (M-4-7) -- node{$<$} (M-4-8) -- node{$<$} (M-4-9)
;
\end{tikzpicture}
\]
The total length of the represented broken chain is then equal to the number of $<$ and $\sim$ signs put together.
\end{definition}

Now we state our main theorem:

\begin{theorem}\label{thmOrbitsBrokenChains}
Let $\cF$ be a saturated fusion system over a finite $p$-group $S$. Let $\BC_\cF (Q,P)$ denote the set of $\cF$-broken chains linking $Q$ to $P$. Then the number of $[S/Q]$-orbits in each irreducible $\cF$-stable set $\alpha_{P}$, denoted $c_Q(\alpha_P)$, can be calculated as
\[c_Q(\alpha_P) = \frac{\abs{W_S P^*}}{\abs{W_S Q}}\cdot \sum_{\sigma\in \BC_\cF(Q,P)} (-1)^{\ell(\sigma)}\]
for $Q,P\in \cP$, where $P^*\sim_\cF P$ is fully normalized.
\end{theorem}

\begin{proof}
Immediate from the argument at the beginning of the section.
\end{proof}

\begin{remark}\label{rmkTetheredIsASpecialCase}
If a broken chain $(\sigma_0,\sigma_1,\dotsc,\sigma_k)\in \BC_\cF(Q,P)$ happens to have $\sigma_0$ equal to the trivial chain, i.e. $\abs{\sigma_0}=0$, then $Q$ is the endpoint of $\sigma_0$ so $\sigma_1$ has to start at $Q^*$. The converse is also true, if $\sigma_1$ starts at $Q^*$, then $\sigma_0$ has to be trivial. In this case $(\sigma_1,\dotsc,\sigma_k)$ is exactly the data of a \emph{tethered} broken chain linking $Q^*$ to $P$.

Hence the tethered broken chains $(\sigma_1,\dotsc,\sigma_k)\in \tBC_\cF(Q^*,P)$ correspond precisely to the broken chains $(\sigma_0,\sigma_1,\dotsc,\sigma_k)\in \BC_\cF(Q,P)$ where $\sigma_0$ is the trivial chain.
This way, in diagram form, a tethered broken chain linking $Q$ (or rather $Q^*$) to $P$ looks like
\[
\begin{tikzpicture}
\matrix (M) [matrix of math nodes, column sep=.5cm, row sep=.5cm] {
Q &&&&&[-.8cm] &[-.8cm] && \\
Q^* & \dotsb & a^1_{n_1} &&&& && \\
&& a^2_0 &\dotsb &a^2_{n_2} &&&& \\[-.3cm]
&&&&\phantom{a^3_0}& \ddots &\phantom{a^p_q}&&\\[-.3cm]
&&&&&& a^k_0 & \dotsb & a^k_{n_k} \\
};
\path[sloped]
(M-1-1) -- node{$\sim$}
(M-2-1) -- node{$<$} (M-2-2) -- node{$<$} (M-2-3) -- node{$\sim$} (M-3-3) -- node{$<$} (M-3-4) -- node{$<$} (M-3-5) -- node{$\sim$} (M-4-5)
(M-4-7) -- node{$\sim$} (M-5-7) -- node{$<$} (M-5-8) -- node{$<$} (M-5-9)
;
\end{tikzpicture}
\]
with $a^k_{n_k}\sim_\cF P$. Drawn in this form, the total length of the tethered broken chain is the total number of $<$ and $\sim$ symbols, where the initial $Q\sim Q^*$ adds the necessary $+1$ in comparison with Definition \ref{defTetheredBrokenChain}.
\end{remark}

Theorem \ref{thmFixedPointBrokenChains} can thus be reformulated as
\begin{cor}[Theorem \ref{thmFixedPointBrokenChains} revisited]
Let $\cF$ be a saturated fusion system over a finite $p$-group $S$. The numbers of fixed points for each irreducible $\cF$-stable set $\alpha_{P}$ can be calculated as
\[\abs{(\alpha_P)^Q} = \abs{W_S P^*}\cdot \sum_{\substack{\sigma=(\sigma_0,\dotsc,\sigma_k)\in \BC_\cF(Q,P)\\ \abs{\sigma_0}=0}} (-1)^{\ell(\sigma)}\]
for $Q,P\in \cP$, where $P^*\sim_\cF P$ is fully normalized.
\end{cor}

\begin{example}\label{exBrokenChains}  Let $S=D_8$ and $\cF=\cF_S(A_6)$ as before. We showed earlier that $c_Q (\alpha_P)=1$ when $Q=C_2^1$ and $P=V_4^2$. Note that in this case $\abs{W_S P^*}=\abs{W_SQ}=2$ and there is only one broken chain from $C_2^1$ to $V_4^2$ which is
\[
\begin{tikzpicture}[ampersand replacement=\&,baseline=(M-1-1.base)]
\matrix (M) [matrix of math nodes, column sep=.5cm, row sep=.5cm] {
  C_2^1 \&\\
  Z  \&  V_4^2 \\
};
\path[sloped]
(M-1-1)   -- node{$\sim$} (M-2-1) -- node{$<$} (M-2-2);
\end{tikzpicture}
\]
Note that this is also a tethered broken chain. So we have $\abs{(\alpha_P )^Q}=\abs{W_S P^*}\cdot 1=2$ for $Q=C_2 ^1$ and $P=V_4 ^2$.

If we repeat the same calculation for $Q=C_2^1$ and $P=D_8$, then we observe that there are $10$ broken chains from $C_2^1$ to $D_8$ which are
\[
\begin{tikzpicture}[ampersand replacement=\&,baseline=(M-1-1.base)]
\matrix (M) [matrix of math nodes, column sep=.5cm, row sep=.5cm] {
C_2 ^1 \& D_8 \&  \\
C_2^1 \& V_4^1 \& D_8\\
};
\path[sloped]
(M-1-1)   -- node{$<$} (M-1-2) -- (M-2-1) -- node{$<$} (M-2-2) -- node{$<$} (M-2-3) ;
\end{tikzpicture}
\qquad\qquad
\begin{tikzpicture}[ampersand replacement=\&,baseline=(M-1-1.base)]
\matrix (M) [matrix of math nodes, column sep=.5cm, row sep=.5cm] {
  C_2^1 \&\\
  Z  \&  D_8 \\
};
\path[sloped]
(M-1-1)   -- node{$\sim$} (M-2-1) -- node{$<$} (M-2-2);
\end{tikzpicture}
\qquad\qquad
\begin{tikzpicture}[ampersand replacement=\&,baseline=(M-1-1.base)]
\matrix (M) [matrix of math nodes, column sep=.5cm, row sep=.5cm] {
  C_2^1 \& V_4^1 \& \\
   \&  V_4 ^1 \& D_8 \\
};
\path[sloped]
(M-1-1)   -- node{$<$} (M-1-2) -- node{$\sim$} (M-2-2) -- node{$<$} (M-2-3);
\end{tikzpicture}
\]

\[
\begin{tikzpicture}[ampersand replacement=\&,baseline=(M-1-1.base)]
\matrix (M) [matrix of math nodes, column sep=.5cm, row sep=.5cm] {
  C_2^1 \&  \& \\
  Z \& V_4^1 \& D_8  \\
};
\path[sloped]
(M-1-1)   -- node{$\sim$} (M-2-1) -- node{$<$} (M-2-2) -- node{$<$} (M-2-3) ;
\end{tikzpicture}
\qquad\quad
\begin{tikzpicture}[ampersand replacement=\&,baseline=(M-1-1.base)]
\matrix (M) [matrix of math nodes, column sep=.5cm, row sep=.5cm] {
  C_2^1 \&  \& \\
  Z \& C_4 \& D_8  \\
};
\path[sloped]
(M-1-1)   -- node{$\sim$} (M-2-1) -- node{$<$} (M-2-2) -- node{$<$} (M-2-3) ;
\end{tikzpicture}
\qquad\quad
\begin{tikzpicture}[ampersand replacement=\&,baseline=(M-1-1.base)]
\matrix (M) [matrix of math nodes, column sep=.5cm, row sep=.5cm] {
  C_2^1 \&  \& \\
  Z \& V_4^2 \& D_8  \\
};
\path[sloped]
(M-1-1)   -- node{$\sim$} (M-2-1) -- node{$<$} (M-2-2) -- node{$<$} (M-2-3) ;
\end{tikzpicture}
\]

\[
\begin{tikzpicture}[ampersand replacement=\&,baseline=(M-1-1.base)]
\matrix (M) [matrix of math nodes, column sep=.5cm, row sep=.5cm] {
  C_2^1 \&  \& \\
  Z \& V_4^1 \& \\
 \& V_4 ^1 \& D_8 \\  };
\path[sloped]
(M-1-1) -- node{$\sim$} (M-2-1) -- node{$<$} (M-2-2) -- node{$\sim$} (M-3-2) -- node{$<$} (M-3-3) ;
\end{tikzpicture}
\qquad\quad
\begin{tikzpicture}[ampersand replacement=\&,baseline=(M-1-1.base)]
\matrix (M) [matrix of math nodes, column sep=.5cm, row sep=.5cm] {
  C_2^1 \&  \& \\
  Z \& C_4 \& \\
 \& C_4 \& D_8 \\  };
\path[sloped]
(M-1-1) -- node{$\sim$} (M-2-1) -- node{$<$} (M-2-2) -- node{$\sim$} (M-3-2) -- node{$<$} (M-3-3) ;
\end{tikzpicture}
\qquad\quad
\begin{tikzpicture}[ampersand replacement=\&,baseline=(M-1-1.base)]
\matrix (M) [matrix of math nodes, column sep=.5cm, row sep=.5cm] {
  C_2^1 \&  \& \\
  Z \& V_4^2 \& \\
 \& V_4 ^2 \& D_8 \\  };
\path[sloped]
(M-1-1) -- node{$\sim$} (M-2-1) -- node{$<$} (M-2-2) -- node{$\sim$} (M-3-2) -- node{$<$} (M-3-3) ;
\end{tikzpicture}
\]

If we sum the signs $(-1)^{\ell(\sigma)}$ over  all the broken chains above, and multiply it with $\abs{W_S P^*}/\abs{W_S Q}$, we get
\[ c_Q(\alpha_P)=\frac{1}{2}(1-2+4-3)=0.\] Note that if we only consider the tethered broken chains, then we obtain \[\abs{(\alpha_P)^Q}=\abs{W_SP^*} (1-3+3)=1.\]
\end{example}

Note that in the above example some of the broken chains naturally can be paired with each other to cancel their contributions. For example, all the broken chains on the second row cancels with the broken chains on the third row. In the next section we prove that the broken chain calculations for calculating $c_Q(\alpha_P)$ and $\abs{(\alpha_P )^Q}$ can be simplified.

\section{Computational simplifications}\label{secSimplifications}

In this section, we show that certain types of broken chains can be naturally paired with certain other types of broken chains in such a way that their contributions in the summation in Theorem \ref{thmOrbitsBrokenChains} cancel each other. This gives a modified version of the formula in Thereom \ref{thmOrbitsBrokenChains} where we only consider broken chains which are not in either type. We start with a definition of these types.

\begin{definition}
Let $\sigma=(\sigma_0,\dotsc,\sigma_k)$ be a broken chain in $\cF$ with $\sigma_i=(a^i_0,\dotsc, a^i_{n_i})$.
Suppose that a subgroup $a^i_j$ in the broken chain is $S$-conjugate to the chosen representative $(a^i_j)^*\in \cP^*$. We say that such an $a^i_j$ is \emph{a $*$-group of type 1} if $0<j<n_i$, or if $i=j=0$ and $n_0>0$.
We say that $a^i_j$ is \emph{a $*$-group of type 2} if $j=n_i$ and $0\leq i<k$. In the remaining cases we either have $j=0$ and $i>0$, in which case $a^i_j\in \cP^*$ is always required, or we have $i=k$ and $j=n_k$ with $a^i_j$ as the very last group. In either of these last cases, $a^i_j$ is \emph{not} a $*$-group.

In diagram form the two types of $*$-groups are as follows:\begin{align*}
\tag{\text{Type $1$}}
\begin{tikzpicture}[ampersand replacement=\&,baseline=(M-2-3.base)]
\matrix (M) [matrix of math nodes, column sep=.5cm, row sep=.5cm] {
\ddots \&[-.5cm] \&\&\&[-.5cm] \\[-.8cm]
\& \dotsb \& (a^i_j)^* \& \dotsb \& \\[-.8cm]
\&\&\&\&\ddots\\
};
\path[sloped]
(M-2-2) -- node{$<$} (M-2-3) -- node{$<$} (M-2-4)
;
\end{tikzpicture}
&&\text{or} &&
\begin{tikzpicture}[ampersand replacement=\&,baseline=(M-1-1.base)]
\matrix (M) [matrix of math nodes, column sep=.5cm, row sep=.5cm] {
(a^0_0)^*  \& \dotsb \&[-.5cm] \\[-.8cm]
\&\&\ddots\\
};
\path[sloped]
(M-1-1) -- node{$<$} (M-1-2)
;
\end{tikzpicture}
\\
\tag{\text{Type $2$}}
\begin{tikzpicture}[ampersand replacement=\&,baseline=(M-2-3.base)]
\matrix (M) [matrix of math nodes, column sep=.5cm, row sep=.5cm] {
\ddots \&[-.5cm] \&\&\&[-.5cm] \\[-.8cm]
\& \dotsb \& (a^i_j)^* \&  \& \\
\&\& a^{i+1}_0\& \dotsb \& \\[-.8cm]
\&\&\&\&\ddots\\
};
\path[sloped]
(M-2-2) -- node{$<$} (M-2-3) -- node{$\sim$} (M-3-3) -- node{$<$} (M-3-4)
;
\end{tikzpicture}
&&\text{or} &&
\begin{tikzpicture}[ampersand replacement=\&,baseline=(M-1-1.base)]
\matrix (M) [matrix of math nodes, column sep=.5cm, row sep=.5cm] {
(a^0_0)^* \& \&[-.5cm] \\
a^1_0\& \dotsb \& \\[-.8cm]
\& \&\ddots\\
};
\path[sloped]
(M-1-1)-- node{$\sim$} (M-2-1) -- node{$<$} (M-2-2)
;
\end{tikzpicture}
\end{align*}

If a broken chain $\sigma$ contains at least one $*$-group, we say that $\sigma$ is \emph{sparkling of type 1 or 2} where the type of $\sigma$ is determined by the type of the smallest $*$-group in $\sigma$.
A broken chain is \emph{drab} if it has no $*$-groups at all.
\end{definition}

\begin{example}\label{exTypes}
Consider the last calculation in Example \ref{exBrokenChains}, where $Q=C_2^1$ and $P=D_8$. The broken chains on the second row are all sparkling of type $1$. More specifically in all these, the second chains include $*$-groups of type 1 which are $V_4^1$, $C_4$, and $V_4^2$. Note also that the second broken chain on the first row is a sparkling broken chain of the type $1$. The fourth chain on the first row and all the chains on the third row are sparkling broken chains of type 2. We will see below that these type 1 and type 2 chains can be paired  in an obvious way. The only drab broken chains in this example are the first and third broken chains on the first row.
\end{example}

\begin{prop}\label{propNoStars}
Let $\cF$ be a saturated fusion system over a finite $p$-group $S$. In calculating the coefficients $c_Q(\alpha_P)$ by Theorem \ref{thmOrbitsBrokenChains}, it is sufficient to consider only the drab broken chains:
\[c_Q(\alpha_P) = \frac{\abs{W_S P^*}}{\abs{W_S Q}}\cdot \sum_{\substack{\sigma\in \BC_\cF(Q,P)\\ \sigma\textup{ is drab}}} (-1)^{\ell(\sigma)}\]
for $Q,P\in \cP$, where $P^*\sim_\cF P$ is fully normalized.
\end{prop}

\begin{proof}
By Theorem \ref{thmOrbitsBrokenChains} we have
\[c_Q(\alpha_P) = \frac{\abs{W_S P^*}}{\abs{W_S Q}}\cdot \sum_{\substack{\sigma\in \BC_\cF(Q,P)}} (-1)^{\ell(\sigma)}\]
for $Q,P\in \cP$, where $P^*\sim_\cF P$ is fully normalized. For each subgroup $R\leq S$ we will consider all the sparkling broken chains that have $R$ as their smallest $*$-group and links $Q$ to $P$. For each $R$ we will show that these broken chains cancel each other in the sum above, leaving only the drab broken chains at the end.
In order for $R$ to be a $*$-group at all, $R$ must be $S$-conjugate to the chosen representative $R^*\in \cP^*$. We can therefore choose an $s\in S$ such that $\lc s R = R^*$, and we let $s$ be fixed for the remainder of the proof.

Let $\sigma\in \BC_\cF(Q,P)$ be a broken chain with $R$ as its smallest $*$-group. Suppose $\sigma=(\dotsc,\sigma_*,\dotsc)$ where $\sigma_*$ is the chain containing $R$ as a $*$-group. If $R$ is at the end of $\sigma_*$, then $\sigma$ is type 2, otherwise $\sigma$ is type 1.

If $\sigma$ is type 1, then we write $\sigma_*=(A_0,\dotsc,A_{m-1},R,B_1,\dotsc,B_n)$ where $n\geq 1$. We can then conjugate the entire second part of the chain with $s$ to get subgroups $C_i := \lc s B_i$. These form a chain $(R^*,C_1,\dotsc, C_n)$ which starts at $R^*\in \cP^*$ and has length at least $1$ (see the illustration \eqref{eqSparklingCancellation} below).
We also have $C_n\sim_S B_n\sim_\cF B_n^*$, so we can ``break'' $\sigma_*$ at $R$ into two chains and get a legal broken chain $\sigma':=(\dotsc,(A_0,\dotsc,A_{m-1},R),(R^*,C_1,\dotsc,C_n),\dotsc)$ where we don't change any other part of $\sigma$. The new broken chain $\sigma'$ is type $2$ with $R$ as its smallest $*$-group. Since $\sigma'$ has one extra break compared to $\sigma$, $\ell(\sigma')=\ell(\sigma)+1$.

If alternatively $\sigma$ has type 2, we write $\sigma_*=(A_0,\dotsc,A_{m-1},R)$ and let $(R^*,C_1,\dotsc,C_n)$ be the chain of $\sigma$ that follows $\sigma_*$ (such a chain exists since $R$ is not the very last group of $\sigma$). We conjugate every $C_i$ with $s$ from the right $B_i:=C_i^s$, and they form a chain $(R,B_1,\dotsc,B_n)$ starting at $R$ and satisfying $B_n\sim_S C_n\sim_\cF C_n^*$. We can then combine $\sigma_*$ with the $B_i$-chain to get a single chain, and a new broken chain $\sigma':=(\dotsc,(A_0,\dotsc,A_{m-1},R,B_1,\dotsc,B_n),\dotsc)$ of type $1$ with $R$ as its smallest $*$-group. We also have $\ell(\sigma')=\ell(\sigma)-1$.

The two operations are inverses to each other and are illustrated below:
\begin{equation}\label{eqSparklingCancellation}
\begin{tikzpicture}[baseline=(M-2-4.base)]
\matrix (M) [matrix of math nodes, column sep=.5cm, row sep=.5cm] {
\ddots &&\\
A_0 & \dotsb & A_{m-1} & R & B_1&\dotsb& B_n  \\
&&& R^* & C_1 &\dotsb &C_n && \\
&&&&&& \ddots \\
};
\path[sloped]
(M-1-1) -- node{$\sim$} (M-2-1) -- node{$<$} (M-2-2) -- node{$<$} (M-2-3) -- node{$<$} (M-2-4) -- node{$<$} (M-2-5) -- node{$<$} (M-2-6) --node{$<$} (M-2-7) -- node{$\sim_S$} (M-3-7)
(M-2-4) -- node{$\sim_S$} (M-3-4) -- node{$<$} (M-3-5) -- node{$<$} (M-3-6)-- node{$<$} (M-3-7) -- node{$\sim$} (M-4-7)
(M-2-5) --node{$\sim_S$} (M-3-5)
;
\node [node distance=.2cm, above=of M-2-6] {Type 1};
\node [node distance=.2cm, below=of M-3-4] {Type 2};
\end{tikzpicture}
\end{equation}
Because any two corresponding broken chains have lengths that differ by $1$, they cancel in the sum of the Theorem \ref{thmOrbitsBrokenChains}.
\end{proof}

Another way to reduce the number of terms in the sum of Theorem \ref{thmOrbitsBrokenChains}, is to limit the sizes of the individual chains in a broken chain. This stems from the fact that the usual M\"obius function for subgroups of $p$-groups has $\mu(A,B)=0$ unless $B\leq N_S A$ with $B/A$ elementary abelian (see  \cite{HawkesIsaacsOzaydin}*{Corollary 3.5},
 \cite{KratzerThevenaz}*{Proposition 2.4}).

\begin{prop}\label{propElmAbQuotients}
Let $\cF$ be a saturated fusion system over a finite $p$-group $S$. In calculating the coefficients $c_Q(\alpha_P)$ by Theorem \ref{thmOrbitsBrokenChains}, it is sufficient to consider only broken chains $(\sigma_0,\dotsc,\sigma_k)$ where every $\sigma_i=(a^i_0,\dotsc,a^i_{n_i})$ has $a^i_{n_i}\leq N_S(a^i_0)$ with $a^i_{n_i}/a^i_0$ elementary abelian. Therefore, we have
\[c_Q(\alpha_P) = \frac{\abs{W_S P^*}}{\abs{W_S Q}}\cdot \sum_{\substack{\sigma=((a^i_j)_{j=0}^{n_i})_{i=0}^k\in \BC_\cF(Q,P),\\ \textup{s.t. each $a^i_{n_i}/a^i_0$ is elm.ab.}}} (-1)^{\ell(\sigma)}\]
for $Q,P\in \cP$, where $P^*\sim_\cF P$ is fully normalized.
\end{prop}

\begin{proof}
In the proof of Theorem \ref{thmOrbitsBrokenChains} we consider the sum
\[c_Q(\alpha_P)= \frac{\abs{W_S P^*}}{\abs{W_S Q}}\hspace{-.3cm} \sum_{\substack{R_0, R_1,\dotsc,R_k\in \cP\\ \text{s.t. } R_k\sim_\cF P^*,\\ \abs{Q}\leq\abs{R_0}<\abs{R_1}<\dotsb <\abs{R_k}}}\hspace{-.3cm} (-1)^k\mu(Q,R_0)\mu(R_0^*,R_1)\mu(R_1^*,R_2)\dotsm\mu(R_{k-1}^*,R_k)\]
A term of this sum is only nonzero if $Q\lhd R_0$ and $R_{i-1}\lhd R_i$ with elementary abelian quotients for all $i$. Hence the sum reduces to
\[c_Q(\alpha_P)= \frac{\abs{W_S P^*}}{\abs{W_S Q}}\hspace{-.6cm} \sum_{\substack{R_0, R_1,\dotsc,R_k\in \cP\\ \text{s.t. } R_k\sim_\cF P^*,\\ \abs{Q}\leq\abs{R_0}<\abs{R_1}<\dotsb <\abs{R_k}, \\ \textup{$R_0/Q$ and $R_i/R_{i-1}$ are elm.ab.}}}\hspace{-.6cm} (-1)^k\mu(Q,R_0)\mu(R_0^*,R_1)\mu(R_1^*,R_2)\dotsm\mu(R_{k-1}^*,R_k)\]
As in the proof of Theorem \ref{thmOrbitsBrokenChains} we then replace each product of M\"obius functions by broken chains and arrive at the formula in the proposition.
\end{proof}

\begin{remark}
Sadly the two reductions of Propositions \ref{propNoStars} and \ref{propElmAbQuotients} cannot be combined, as that would require cancelling the same broken chain with two different other broken chains. To see this, let $\cF=\cF_{D_8}(A_6)$ be as in Example 3.2, where we showed that $\alpha_{D_8}=[S/D_8]$.  Let us show that if we exclude both the sparkling broken chains and those that violate the hypothesis of Proposition \ref{propElmAbQuotients}, then we would not be able to compute the coefficient of the orbit $[S/C_2^1]$ in $\alpha_{D_8}$ correctly.

As it is listed in Example \ref{exBrokenChains}, there are a total of 10 broken chains linking $C_2^1$ to $D_8$.  Of these, only $(C_2^1<D_8)$ and $(C_2^1,Z<D_8)$ are drab, and of those, only the second would be counted in Proposition \ref{propElmAbQuotients}.  Thus there is no chance for cancelation, and the intersections of Propositions \ref{propNoStars} and \ref{propElmAbQuotients} would yield $c_{C_2^1}(\alpha_{D_8})=1/2$, which is obviously false.  The issue is that there can be cancelation between sparkling subgroups and subgroups that violate the hypothesis of Proposition \ref{propElmAbQuotients}, so that by combining both conditions we may undercount the cancelations needed in the proof of Theorem 5.2.
\end{remark}

\section{An application to characteristic bisets}
\label{secApplications}

In this section we demonstrate how we can use Theorem \ref{thmOrbitsBrokenChains} to give structural results for the minimal characteristic biset associated to a saturated fusion system.

\begin{definition}\label{defCharBiset}
We consider $(S,S)$-bisets, i.e. finite sets equipped with both a left $S$-action and a right $S$-action, and such that the actions commute. The structure of such a biset $X$ is equivalent to an action of $S\x S$ on $X$ with $(s_1,s_2).x = s_1.x.(s_2)^{-1}$, and for each point $x\in X$ we speak of the stabilizer $Stab_{S\x S}(x)$ as a subgroup of $S\x S$.

An \emph{$\cF$-characteristic biset} for a fusion system $\cF$ on $S$ is a biset $\Omega$ satisfying three properties originally suggested by Linckelmann-Webb:
\begin{enumerate}
\item\label{itemF-generated} For every point $\omega\in \Omega$ the stabilizer $Stab_{S\x S}(\omega)$ has the form of a \emph{graph}/\emph{twisted diagonal} $\Delta(P,\ph)$ for some $\ph\in \cF(P,S)$ and $P\leq S$, where the twisted diagonal $\Delta(P,\ph)\leq S\x S$ is defined as
    \[\Delta(P,\ph) = \{(\ph(s),s) \mid s\in P\}.\]
\item\label{itemF-stable} $\Omega$ is $\cF$-stable with respect to both $S$-actions. For bisets that satisfy property \ref{itemF-generated} this boils down to checking that the number of fixed points satisfy
\[\abs*{\Omega^{\Delta(P,id)}} = \abs*{\Omega^{\Delta(P,\ph)}} = \abs*{\Omega^{\Delta(\ph P, id)}}\]
for all $\ph\in \cF(P,S)$ and $P\leq S$.
\item The prime $p$ does not divide $\abs \Omega / \abs S$ (which is an integer because of \ref{itemF-generated}). This ensures that $\Omega$ is not degenerate.
\end{enumerate}
\end{definition}

In \cite{RagnarssonStancu} it is shown that the exists a characteristic biset for $\cF$ if and only if $\cF$ is saturated, and it is shown how to reconstruct $\cF$ given any $\cF$-characteristic biset. In \cite{GelvinReeh} two of the authors of this paper give a parametrization of all the characteristic bisets for a given saturated fusion system $\cF$. In particular it is shown that there is a unique minimal $\cF$-characteristic biset $\Lambda_\cF$, and every other $\cF$-characteristic biset contains at least one copy of $\Lambda_\cF$.
\begin{theorem}[\cite{GelvinReeh}*{Theorem 5.3 and Corollary 5.4}]
Let $\cF$ be a saturated fusion system on a finite $p$-group $S$, and consider the product fusion system $\cF\x \cF$ on $S\x S$. According to Proposition \ref{propDefiningAlphas} there is an irreducible $(\cF\x \cF)$-stable $(S\x S)$-set $\alpha_{\Delta(S,id)}$ corresponding to the diagonal $\Delta(S,id)\leq S\x S$. Denote this $(S\x S)$-set or $(S,S)$-biset by $\Lambda_\cF:= \alpha_{\Delta(S,id)}$.

The biset $\Lambda_\cF$ is then $\cF$-characteristic, and every $\cF$-characteristic biset contains a copy of $\Lambda_\cF$ (up to isomorphism). Hence $\Lambda_\cF$ is the unique minimal characteristic biset for $\cF$.
\end{theorem}

In order to apply Theorem \ref{thmOrbitsBrokenChains} to study $\Lambda_\cF$ we need to figure out what broken chains look like in the context of bisets and the fusion system $\cF\x \cF$.

In a product fusion system the conjugation is defined coordinatewise. Hence two twisted diagonals $\Delta(P,\ph)$ and $\Delta(P',\ph')$ are conjugate in $\cF\x \cF$ if and only if there are additional isomorphisms $\psi,\rho\in \cF$ such that $\ph'=\psi\circ \ph\circ \rho^{-1}$. Consequently, every $\Delta(P,\ph)$ with $\ph\in \cF(P,S)$ is conjugate to $\Delta(P,id)$ which is conjugate to $\Delta(P',id)$ for all $P'\sim_\cF P$. In addition the subgroups of $S\x S$ that are subconjugate to $\Delta(S,id)$ in $\cF\x \cF$ are precisely all the twisted diagonals $\Delta(P,\ph)$ with $\ph\in \cF(P,S)$ and $P\leq S$.
To study $\Lambda_\cF=\alpha_{\Delta(S,id)}$ we therefore have to consider broken chains where all the groups are twisted diagonals coming from maps in $\cF$.

Two twisted diagonals satisfy $\Delta(Q,\psi)\leq \Delta(P,\ph)$ exactly when $\ph$ extends $\psi$, i.e. $Q\leq P$ and $\psi=\ph|_Q$. Every $(\cF\x \cF)$-conjugacy class of twisted diagonals contains a fully normalized representative on the form $\Delta(P^*,id)$ where $P^*$ is fully $\cF$-normalized, suppose for Theorem \ref{thmOrbitsBrokenChains} that we have chosen such at fully normalized representative $\Delta(P^*,id)$ for each conjugacy class.
The broken chains that we consider are chains of inclusions connected by $(\cF\x\cF)$-conjugations.
\begin{itemize}
\item Every chain of inclusions $\Delta(P_1,\ph_1)\leq \dotsb\leq \Delta(P_k,\ph_k)$ is a sequence of extensions with $\ph_i= \ph_k|_{P_i}$.
\item Every chain (except for the $0$'th chain) starts with a diagonal of the form $\Delta(P^*,id)$ where $P^*$ a fully normalized representative for the $\cF$-conjugacy class.
\end{itemize}
With this insight we can now apply Theorem \ref{thmOrbitsBrokenChains} and relate $\Lambda_\cF$ to the largest normal subgroup in $\cF$. Here normality is in the sense of \cite{AKO}*{Definition 4.3} where $P\leq S$ is normal in $\cF$ if every homomorphism $\ph\in \cF(Q,R)$ extends to some $\tilde\ph \in \cF(QP,RP)$ with $\tilde\ph(P)=P$. For each fusion system $\cF$ there is a largest normal subgroup, denoted $O_p(\cF)$.

\begin{prop}[\cite{GelvinReeh}*{Proposition 9.11}]\label{proApplication}
Let $\cF$ be a saturated fusion system on a finite $p$-group $S$, and let $\Lambda_\cF$ be the minimal characteristic biset for $\cF$. Denote by $O_p(\cF)$ the largest normal subgroup of $\cF$.
Then for each point $\omega\in \Lambda_\cF$ the stabilizer $Stab_{S\x S}(\omega)=\Delta(P,\ph)$ satisfies $P\geq O_p(\cF)$.
\end{prop}

The original proof in \cite{GelvinReeh} is quite involved. In contrast the proof below, using broken chains, is actually quite elementary once you have the idea of pairing broken chains of opposite sign together.

\begin{proof}
Let $\Delta(R,\rho)$ with $\rho\in \cF(R,S)$  be such that $R$ does not contain $O_p(\cF)$. We then wish to show that $c_{\Delta(R,\rho)}(\Lambda_\cF)=0$. Because $\Lambda_\cF=\alpha_{\Delta(S,id)}$, we can apply Theorem \ref{thmOrbitsBrokenChains} and consider all $(\cF\x \cF)$-broken chains linking $\Delta(R,\rho)$ with $\Delta(S,id)$:
\begin{equation}\label{eqCoeffOfMinimalCharBiset}
c_{\Delta(R,\rho)}(\Lambda_\cF) = \sum_{\sigma\in \BC_{\cF\x \cF}(\Delta(R,\rho),\Delta(S,id))} (-1)^{\ell(\sigma)}.
\end{equation}
We will then show that all these broken chains cancel in pairs of two broken chains with opposite signs.

Consider a broken chain $\sigma\in \BC_{\cF\x \cF}(\Delta(R,\rho),\Delta(S,id))$. Since $S$ contains $O_p(\cF)$ and $R$ does not, there is a first twisted diagonal $\Delta(P,\ph)$ in $\sigma$ with $P\geq O_p(\cF)$, and necessarily $P> R$. Note that $\Delta(P,\ph)$ cannot be in the beginning of any chain in $\sigma$, since normality of $O_p(\cF)$ implies that the end of the previous chain would also contain $O_p(\cF)$.

Let $\Delta(Q,\psi)$ be the twisted diagonal coming just before $\Delta(P,\ph)$ in $\sigma$.
Because $O_p(\cF)$ is normal in $\cF$, hence also in $S$ and $P$, the product $Q\cdot O_p(\cF)$ is a well-defined subgroup of $P$. If we restrict $\ph$ to $Q O_p(\cF)$ we then have inclusions
\[\Delta(Q,\psi)< \Delta(Q O_p(\cF),\ph|\,) \leq \Delta(P,\ph).\]
If $P\neq Q O_p(\cF)$, then the broken chain $\sigma$ looks like
\[\begin{tikzpicture}[ampersand replacement=\&,baseline=(M-2-3.base)]
\matrix (M) [matrix of math nodes, column sep=.5cm, row sep=.5cm] {
\ddots \&[-.5cm] \&\&[-.5cm] \\[-.8cm]
\& \Delta(Q,\psi) \& \Delta(P,\ph) \& \\[-.8cm]
\&\& \& \ddots \\
};
\path[sloped]
(M-2-2) -- node{$<$} (M-2-3)
;
\end{tikzpicture}
,
\]
and we can add $\Delta(QO_p(\cF),\ph|\,)$ in the middle to make the broken chain one step longer.
Conversely, if $P=QO_p(\cF)$, and if $\Delta(P,\ph)=\Delta(QO_p(\cF), \varphi)$ is not at the end of a chain in $\sigma$, then $\sigma$ looks like
\[
\begin{tikzpicture}[ampersand replacement=\&,baseline=(M-2-3.base)]
\matrix (M) [matrix of math nodes, column sep=.5cm, row sep=.5cm] {
\ddots \&[-.5cm] \&\&\&[-.5cm] \\[-.8cm]
\& \Delta(Q,\psi) \& \Delta(QO_p(\cF),\ph) \& \Delta(T,\eta) \& \\[-.8cm]
\&\& \&\& \ddots \\
};
\path[sloped]
(M-2-2) -- node{$<$} (M-2-3) -- node{$<$} (M-2-4)
;
\end{tikzpicture}
,
\]
and we can remove $\Delta(P,\ph)=\Delta(QO_p(\cF),\ph)$ to make the broken chain one step shorter. These two constructions are inverse to each other, hence the broken chains with $P\neq Q O_p(\cF)$ are paired with the broken chains where $P=Q O_p(\cF)$ and $\Delta(P,\ph)$ is not at the end of a chain, and the pairing is such that the total length changes by $1$. Hence these broken chains cancel each other in \eqref{eqCoeffOfMinimalCharBiset}, and we are left with the broken chains $\sigma$ where $\Delta(P,\ph)=\Delta(QO_p(\cF),\ph)$ is at the end of a chain in $\sigma$.

All the remaining broken chains look like
\[
\begin{tikzpicture}[ampersand replacement=\&,baseline=(M-2-3.base)]
\matrix (M) [matrix of math nodes, column sep=.5cm, row sep=.5cm] {
\ddots \&[-.5cm] \&\&\&[-.5cm] \\[-.8cm]
\& \Delta(Q,\psi) \& \Delta(QO_p(\cF),\ph) \&\& \\
\& \& \Delta(T,\eta)  \&\dotsb\&\\[-.8cm]
\&\& \&\& \ddots \\
};
\path[sloped]
(M-2-2) -- node{$<$} (M-2-3) -- node{$\sim$} (M-3-3) -- node{$<$} (M-3-4)
;
\end{tikzpicture}
\]
Let $\sigma_i$ be the chain in $\sigma$ that contains the segment $\Delta(Q,\psi)<\Delta(QO_p(\cF),\ph)$. We divide the remaining broken chains into two types: Those broken chains where $\sigma_i$ consists only of $\Delta(Q,\psi)<\Delta(QO_p(\cF),\ph)$ and has $i\geq 1$; we call these Type A. The remaining broken chains form Type B, i.e. the broken chains where the chain $\sigma_i$ contains twisted diagonals before $\Delta(Q,\psi)$, or where $i=0$. We will finish the proof by cancelling broken chains of Type A with those of Type B and vice versa.

For each possible choice of $Q$, there is a chosen representative $Q^*\sim_\cF Q$ such that $\Delta(Q^*,id)$ is fully normalized in the $(\cF\x \cF)$-conjugacy class of $\Delta(Q,id)$ and $\Delta(Q,\psi)$ for all $\psi\in \cF(Q,S)$. For each $Q\sim_\cF Q^*$, we choose a particular $\cF$-isomorphism $\chi_Q\colon Q\to Q^*$. For each $Q\sim_\cF Q^*$ and each homomorphism $\psi\in \cF(Q,S)$, we also make the choice of an extension $\tilde\psi \colon QO_p(\cF) \to \psi(Q)O_p(\cF)$ in $\cF$ such that $\tilde\psi|_Q=\psi$ and $\tilde\psi(O_p(\cF))=O_p(\cF)$. In particular, we have isomorphisms $\tilde\chi_Q\colon QO_p(\cF)\to Q^*O_p(\cF)$.

If $\sigma$ is Type A, then by the definition of broken chains we must have $\Delta(Q,\psi) =\Delta(Q^*,id)$  because it is the start of the chain $\sigma_i$ and $i\geq 1$. Hence $\sigma$ looks like
\[\tag{\text{Type A}}
\begin{tikzpicture}[ampersand replacement=\&,baseline=(M-3-3.base)]
\matrix (M) [matrix of math nodes, column sep=.5cm, row sep=.5cm] {
\ddots \&[-.5cm]  \&\&\&[-.5cm] \\[-.8cm]
\& \Delta(Q',\psi') \&\&\& \\
\&  \Delta(Q^*,id) \& \Delta(Q^*O_p(\cF),\ph) \&\& \\
\&  \& \Delta(T,\eta)  \&\dotsb\&\\[-.8cm]
\& \&\&\& \ddots \\
};
\path[sloped]
(M-2-2) -- node{$\sim$} (M-3-2) -- node{$<$} (M-3-3) -- node{$\sim$} (M-4-3) -- node{$<$} (M-4-4)
;
\end{tikzpicture}
\]
We pair this with the following chain of Type B and a total length that has decreased by $1$:
\[\tag{\text{Type B}}
\begin{tikzpicture}[ampersand replacement=\&,baseline=(M-2-3.base)]
\matrix (M) [matrix of math nodes, column sep=.5cm, row sep=.5cm] {
\ddots \&[-.5cm]  \&\&\&[-.5cm] \\[-.8cm]
\& \Delta(Q',\psi') \& \Delta(Q'O_p(\cF),\tilde\psi' \circ (\tilde\chi_{Q'})^{-1}\circ \ph \circ \tilde\chi_{Q'}) \&\& \\
\& \&  \Delta(T,\eta)  \&\dotsb\&\\[-.8cm]
\&\& \&\& \ddots \\
};
\path[sloped]
(M-2-2)  -- node{$<$} (M-2-3) -- node{$\sim$} (M-3-3) -- node{$<$} (M-3-4)
;
\end{tikzpicture}
\]
Conversely, if $\sigma$ is Type B, then it has the shape
\[\tag{\text{Type B}}
\begin{tikzpicture}[ampersand replacement=\&,baseline=(M-2-3.base)]
\matrix (M) [matrix of math nodes, column sep=.5cm, row sep=.5cm] {
\ddots \&[-.5cm]  \&\&\&[-.5cm] \\[-.8cm]
\& \Delta(Q,\psi) \& \Delta(QO_p(\cF),\ph) \&\& \\
\& \&  \Delta(T,\eta)  \&\dotsb\&\\[-.8cm]
\&\& \&\& \ddots \\
};
\path[sloped]
(M-2-2)  -- node{$<$} (M-2-3) -- node{$\sim$} (M-3-3) -- node{$<$} (M-3-4)
;
\end{tikzpicture}
\]
and we can split $\sigma_i$ into two chains, thereby increasing the total length by $1$:
\[\tag{\text{Type A}}
\begin{tikzpicture}[ampersand replacement=\&,baseline=(M-3-3.base)]
\matrix (M) [matrix of math nodes, column sep=.5cm, row sep=.5cm] {
\ddots \&[-.5cm]  \&\&\&[-.5cm] \\[-.8cm]
\& \Delta(Q,\psi) \&  \&\& \\
\& \Delta(Q^*,id) \& \Delta(Q^*O_p(\cF),\tilde\chi_{Q}\circ(\tilde\psi)^{-1}\circ\ph\circ(\tilde\chi_Q)^{-1}) \&\& \\
\& \&  \Delta(T,\eta)  \&\dotsb\&\\[-.8cm]
\&\& \&\& \ddots  \\
};
\path[sloped]
(M-2-2) -- node{$\sim$} (M-3-2) -- node{$<$} (M-3-3) -- node{$\sim$} (M-4-3) -- node{$<$} (M-4-4)
;
\end{tikzpicture}\]
When we split $\sigma_i$ this way, $\Delta(Q,\psi)$ is still in a chain of length at least $1$ if $i\geq 1$, and if $i=0$, then $\Delta(Q,\psi)$ is allowed to form a trivial chain by itself.

This completes the proof as all the remaining broken chains of Type A cancel in \eqref{eqCoeffOfMinimalCharBiset} with all those of Type B.
\end{proof}

\makeatletter
\def\eprint#1{\@eprint#1 }
\def\@eprint #1:#2 {%
    \ifthenelse{\equal{#1}{arXiv}}%
        {\href{http://front.math.ucdavis.edu/#2}{arXiv:#2}}%
        {\href{#1:#2}{#1:#2}}%
}
\makeatother

\end{document}